\documentclass[12pt]{amsart}
\usepackage{amssymb}
\usepackage[active]{srcltx}
\usepackage{epsfig}
\usepackage{amscd}
\usepackage{a4wide}
\usepackage{amsthm}
\usepackage{tikz}
\usepackage{tikz-cd} 
\usetikzlibrary{matrix}
\newtheorem{theorem}{Theorem}[section]
\newtheorem{lemma}[theorem]{Lemma}
\newtheorem{proposition}[theorem]{Proposition}
\newtheorem{corollary}[theorem]{Corollary}

\theoremstyle{definition}
\newtheorem{definition}[theorem]{Definition}
\newtheorem{example}[theorem]{Example}
\newtheorem{remark}[theorem]{Remark}

\numberwithin{equation}{section}

   \DeclareMathOperator{\Sphere}{S}

\newcommand{\bc}{{\mathbb C}}

\newcommand{\bz}{{\mathbb Z}}

\begin{document}
\title [Parabolic representations of 2-bridge Links] {Symplectic quandles and parabolic representations of 2-bridge Knots and Links}
\author{Kyeonghee Jo and Hyuk Kim}
  \address{Division of Liberal Arts and sciences,
  Mokpo National Maritime University, Mokpo, Chonnam, 530-729, Korea  }
\email{khjo@mmu.ac.kr}
\address{Department of Mathematical Sciences, Seoul National University, Seoul, 08826, Korea}
\email{hyukkim@snu.ac.kr}
\thanks {The second author was supported 
 by the National Research Foundation of Korea(NRF) grant funded by the Korea government(MSIT) (NRF-2018R1A2B6005691)}

\begin{abstract}
In this paper we study the  parabolic representations of 2-bridge links by finiding arc coloring vectors on the Conway diagram. The method we use is to convert the system of conjugation quandle equations to that of symplectic quandle equations. 
In this approach, 
we have an integer coefficient monic polynomial $P_K(u)$ for each 2-bridge link $K$, and each zero of this polynomial gives a set of arc coloring vectors on the diagram of $K$ satisfying  the system of symplectic quandle equations, which gives an explicit formula for a parabolic representation of $K$. We then explain how these arc coloring vectors give us the closed form formulas of the complex volume and the cusp shape of the representation. As other  applications of this method, we show some interesting arithmetic properties of the Riley polynomial and of the trace field, and also describe a necessary and sufficient condition for the existence of epimorphisms between 2-bridge link groups in terms of divisibility of the corresponding Riley polynomials.
\end{abstract}
\maketitle
\tableofcontents
\section{Introduction}
Since the volume conjecture connects the quantum invariants of a knot and the hyperbolic geometry of the knot complement, it has attracted a lot of attentions in the past two decades. (See for instance a recent book by Murakami and Yokota  \cite{MY}  for an introduction to the subject.) Also it was further generalized by Gukov, using $SL(2,\mathbb C)$ Chern-Simons theory, into the form of an asymptotic expansion of the colored Jones polynomial whose leading term is the complex volume and a subleading term is essentially the Reidemeister torsion \cite{G}. The real part of the complex volume is the hyperbolic volume and the imaginary part is the Chern-Simons invariant, which are very important invariants of hyperbolic 3-manifolds but in general difficult to compute. The simplicial formula of these invariants using ideal triangulations of hyperbolic 3-manifolds  was set by Neumann \cite{N}, and then more efficient way in the cusped case was given by Zickert using a relative version of the earlier theory \cite{Zickert}.

The method of Neumann and Zickert can be applied to the link complement, and a diagrammatic method using quandle homology was studied by Inoue and Kabaya \cite{IK}. And then Cho and Murakami introduced a combinatorial way of computing the complex volume motivated from the volume conjecture \cite{Cho2, Cho-M} following Yokota's work \cite{Yokota}. In fact they explicitly expressed the complex volume formula in the form of a state sum for a given link diagram, signifying its origin in quantum invariants, using region variables (or called “$w$-variables”), which can be obtained from quandle coloring vectors. This formula is by far the simplest way of describing the complex volume from a diagram of a link in a closed form formula in terms of region variables.  The notion of complex volume was defined for hyperbolic manifolds, but can be generalized to any   $SL(2,\mathbb C)$ representation and we are using this generalized version in this paper. (See \cite{Zickert} for the complex volume of a representation and \cite{GTZ} for the $SL(n,\mathbb C)$ case.) 
For actual computations of the volume, one has to solve a system of algebraic equations, which essentially corresponds to the hyperbolicity equations, to get the region variables \cite{KKY1}. But this system of equations is not easy to solve in general, and instead, one solves for the quandle equations to get the coloring vectors in $\mathbb C^2$, which then gives the region variables immediately by taking the determinant of the coloring vectors  with a generic fixed vector \cite{Cho2, Cho-M}. 
Therefore having an explicit volume formula is essentially the same as an explicit formula for the arc coloring vectors by Cho-Murakami's result, and we describe how to get all these for 2-bridge links in this paper.

Also this coloring vector is defined on each arc of the link diagram and is nothing but a short hand notation for a parabolic element in $SL(2,\mathbb C)$ (see Section 3), giving an explicit description of a parabolic representation of Wirtinger generators, and the volume above is the volume of this representation.  Once we have a $PSL(2,\mathbb C)$ representation, we have a pseudo-hyperbolic structure \cite{KKY1, ST} and can talk about hyperbolic invariants such as complex volume, and another such invariant is a cusp shape. This of course can be obtained by calculating the longitude algebraically, but can also be obtained using $w$-variables in the form of an  explicit state sum formula just like the complex volume. And hence the cusp shape of a  parabolic representation also can be obtained once we have a set of arc coloring vectors. (See Section 6.)

Therefore the problem of computing all these quantities boils down to computing arc coloring vectors, which essentially is to find an efficient way of solving the system of conjugation equations determined by Wirtinger relators. 
In general the computation for solving conjugation quandle is very complicated; even the cases for 8-crossing knots do not seem to be completely settled.
And one of the main purpose of this paper is to find an algorithm by ``linearizing'' the system of conjugation equations. If we inspect carefully, we can see that solving this conjugation equation reduces 
to solving  a symplectic quandle equation, which is semi-linear in the sense that linear in one variable and quadratic in the other variable at each crossing.
The notion of the symplectic quandle was  introduced in \cite{Nelson} and \cite{Yetter}, but the relation between conjugation and symplectic quandle  doesn't seem to be considered before.
This change gives us much advantage in carrying out the actual computations as well as conceptual approaches. In this quandle formulation, the determinant of two coloring vectors, called a symplectic form, appears naturally, and this plays an important role in this paper (and this also appears in other papers in different contexts) and will be called a “$u$-variable” in contrast to a $w$-variable.

When we apply the above argument to 2-bridge links, especially in the Conway form, we can solve the symplectic quandle equation and obtain an explicit formula for the arc coloring vectors in terms of $u:=\langle a,b\rangle:=\det(a,b)$, where $a$ and $b$ are two initial coloring vectors at the two bridges. It turns out that the solution $u$ is obtained as  zeros of one single integer coefficient monic polynomial $P(u)$, called the ``rep-polynomial'' in this paper following Riley \cite{Riley1}. Then the solution gives the arc coloring vectors and hence the region variables immediately, and then the complex volume from the formula of Cho-Murakami and the cusp shape from Kim-Kim-Yoon \cite{KKY2}. Here everything is concrete and explicit and can be given in an exact form. 

Needless to say, this polynomial $P(u)$ should be related to the well known Riley polynomial $R(y)$, and indeed it turns out that $\frac{1}{u}P(u)=\pm R(u^2)$. By deriving $P(u)$ from the symplectic quandle, we do not just recover the famous old result back and the coloring vectors, but also we found some interesting unknown arithmetic properties of the Riley polynomial and of the trace field: Namely splitting of the Riley polynomial for the knot case, $R(u^2)=\pm g(u)g(-u)$, for some integral coefficient polynomial $g(u)$ (this remarkable property reminds us Hirasawa-Murasugi conjecture \cite{HM2}) as well as the fact that the trace field is generated by $u$ instead of $y$. Note that $u$ is a square root of $y$.

One good point of using Conway form instead of Schubert form is to turn the diagram upside down to see another two generators of the knot group $G(K)$. We can get the corresponding rep-polynomial $P'(u)$ which has an equal right as $P(u)$, and this observation gives us an interesting application in the epimorphism problem.  In the knot theory, the epimorphism problem has been studied quite extensively  and for the 2-bridge knot case, the problem of characterizing the epimorphism pair, $G(K) \rightarrow G(K')$ in terms of Conway form is essentially settled down by  Ohtsuki-Riley-Sakuma \cite{ORS}
and Aimi-Lee-Sakai-Sakuma \cite{ALS}. (See also    \cite{Suzuki}.)  Recently Kitano and Morifuji proved the existence of an epimorphism when the Riley polynomial, (and hence the rep-ploynomial) of $K'$ divides that of $K$ \cite{KM}. And we show that the converse also holds if we allow both rep-polynomials $P(u)$ and $P'(u)$ in the divisibility, and show that we can generalize to a similar statement for the link case also. (See Section 7.)

The paper is organized as follows. We first setup the symplectic quandle equations and describe the arc coloring vectors as solutions of the equations in Sections 2 and 3, and then discuss $u$-variables and rep-polynomials in Section 4. Then as applications of this approach, we discuss the arithmetic properties mentioned above in Section 5, the complex volume and cusp shape in Section 6, and then the epimorphism problem in Section 7. In appendix we present some explicit formulas of the arc coloring vectors using Chebyshev polynomials.
\section{Preliminaries}
\subsection{2-bridge Knots and Links}
There are  two famous descriptions for 2-bridge knots or links, Scubert's normal form and Conway's  normal form. Together, a knot and link are called a link in this paper, unless we need to specify it.

\noindent {\bf Scubert's normal form} Each 2-bridge link has an associated pair of coprime integers $(\alpha,\beta)$ where $\alpha$ is positive and $\beta$ is an odd integer such that $-\alpha<\beta<\alpha$. We denote the knot or link by  $S(\alpha,\beta)$ and  call it $Schubert's$ $normal$ $form$.  
The followings are well-known facts (see \cite{BZ} or \cite{Kawauchi} for more details).
\begin{itemize}
    \item 
$S(\alpha,\beta)$ is a knot if $\alpha$ is odd and a 2-component link if $\alpha$ is even.
\item
The  mirror of $S(\alpha,\beta)$ is equivalent to $S(\alpha,-\beta)$.
\item
$S(\alpha,\beta)$ and $S(\alpha',\beta')$ are equivalent as  unoriented knots or  links if and only if 
$$\alpha=\alpha', \ \beta'\equiv \beta^{\pm1}\quad (\text{mod} \ \alpha).$$
\end{itemize}
Note that  if we consider a knot (or  a link) and its mirror as equivalent, $S(\alpha,\beta)$ and $S(\alpha',\beta')$ are equivalent knots or  links if and only if 
$$\alpha=\alpha', \ \beta'\equiv \pm \beta^{\pm1}\quad (\text{mod} \ \alpha). $$ 
\noindent {\bf Conway's normal form}
Each 2-bridge link has an associated sequence of non-zero integers $n_1,n_2,\cdots,n_k$, as indicated in  Figure \ref{u-poly},
where $|n_j|$ is the number of crossing contained in the $j$th block and its sign $\epsilon_j=\frac{n_j}{|n_j|}$  is defined as follows : $\epsilon_{2i+1}=1$ or $\epsilon_{2i}=-1$ if the twists of the block are right-handed and $\epsilon_{2i+1}=-1$ or $\epsilon_{2i}=1$ if they are left-handed. We denote the unoriented  2-bridge link with this regular projection by $C[n_1,n_2,\cdots,n_k]$, which is called $Conway's$ $normal$ $form$.  
For example, $C[3,2,3]$ is shown in Figure \ref{PA-Link-diagram}. (See \cite{Conway} or \cite{KL1}.)

Note that such a diagram of a 2-bridge link $C[n_1,n_2,\cdots,n_k]$ corresponds to a  $rational$ $tangle$ with $slope$
$$[n_1,n_2,\cdots,n_k]:=\frac{1}{n_1+\frac{1}{n_2+\frac{1}{\ddots\frac{1}{n_{k-1}+\frac{1}{n_k}}}}}.$$   

Let $\alpha (>0)$ and $\beta$ be coprime integers obtained by the slope of $C[n_1,n_2,\cdots,n_k]$, that is, $$\frac{\beta}{\alpha}=[n_1,n_2,\cdots,n_k].$$
Then it is well-known that  
$C[n_1,n_2,\cdots,n_k]$ is  equivalent to $S(\alpha,\beta)$. Furthermore 
each 2-bridge link $S(\alpha, \beta)$ in Schubert's normal form can be deformed into Conway's normal form $C[n_1,n_2,\cdots,n_k]$  uniquely, if we require that all the $n_i's$ are either positive or negative and $|n_k|\neq 1$.  We will call the unique Conway's normal form the $cannonical$  Conway's normal form of $S(\alpha, \beta)$. 


The following notation will also be  used in this paper : $$J(n_1,n_2,n_3,\cdots,n_k)=:C[n_1,-n_2,n_3,\cdots,(-1)^{k+1}n_k].$$ That is, $n_i$ is always positive whether $i$ is odd or not if the twists of the $i$-th block are right-handed, and negative if the twists are left-handed in $J(n_1,n_2,n_3,\cdots,n_k)$-notation.

Notice that $C[-n_1,-n_2,\cdots,-n_k]$ is  the mirror of $C[n_1,n_2,\cdots,n_k]$ and $$C[(-1)^{k+1}n_k,(-1)^{k+1}n_{k-1},\cdots,(-1)^{k+1}n_1]$$ is  the upside-down of $C[n_1,n_2,\cdots,n_k]$, i.e., the diagram obtained by a half rotation with respect to the horizontal axis. It is well-known that if $\frac{\beta}{\alpha}=[n_1,n_2,\cdots,n_k]$, then 
$$[(-1)^{k+1}n_k,(-1)^{k+1}n_{k-1},\cdots,(-1)^{k+1}n_1]=\frac{\beta'}{\alpha}, \,\, \beta\beta'\equiv 1 \,\,(\text{mod} \ \alpha).$$

\subsection{Riley polynomials of 2-bridge links} 

For a knot or link $K$, the fundamental group of the complement  is called the $knot$ $group$ or the $link$ $group$ and is denoted by $G(K)$. 

The knot group of  a 2-bridge knot $K=S(\alpha,\beta)$ always has a presentation of the form 
\begin{equation}\label{knot-group}
G(K)=\pi_1(\Sphere^3 \setminus K)=\langle a,b \,|\,wa=bw \rangle,
\end{equation}
where $w$ is of the form
\begin{equation}\label{presentation1}
w=a^{\epsilon_1}b^{\epsilon_2}a^{\epsilon_3}b^{\epsilon_4}\cdots a^{\epsilon_{\alpha-2}}b^{\epsilon_{\alpha-1}}
\end{equation}
with each $\epsilon_i=-(-1)^{[i\frac{\beta}{\alpha}]}$.\footnote{This definition coincides with Riley's definition in \cite{Riley1}. } Note that  $\epsilon_i=\epsilon_{\alpha-i}$ and the generators $a$ and $b$ come from the two bridges and represent the meridians.

Suppose that $\rho : G(K)  \rightarrow SL(2,\bc)$ is a non-abelian $parabolic$ representation, i.e., the trace of $\rho$-image of any meridian is 2. Then after conjugating if necessary, we may assume 
\begin{equation}\label{parabolic-rep}
\rho(a)=\begin{pmatrix}
1& 1\\
0 & 1 
 \end{pmatrix} \quad \text{and}\quad
\rho(b)=\begin{pmatrix}
1& 0\\
-y & 1 
 \end{pmatrix}.
\end{equation}
 Riley had shown in \cite{Riley1} that $y$ determines a non-abelian parabolic representation if and only if  $W_{11}=0$. Here $W_{ij}$ is the ($i,j$)-element of 
\begin{equation}\label{parabolic-rep2}
\begin{split}
W&=\rho(w)=\rho(a)^{\epsilon_1}\rho(b)^{\epsilon_2}\rho(a)^{\epsilon_3}\rho(b)^{\epsilon_4}\cdots \rho(a)^{\epsilon_{\alpha-2}}\rho((b)^{\epsilon_{\alpha-1}}\\
&=\begin{pmatrix}
1& 1\\
0 & 1 
 \end{pmatrix}^{\epsilon_1}\begin{pmatrix}
1& 0\\
-y & 1 
 \end{pmatrix}^{\epsilon_2}\begin{pmatrix}
1& 1\\
0 & 1 
 \end{pmatrix}^{\epsilon_3}\cdots \begin{pmatrix}
1& 1\\
0 & 1 
 \end{pmatrix}^{\epsilon_{\alpha-2}}\begin{pmatrix}
1& 0\\
-y & 1 
 \end{pmatrix}^{\epsilon_{\alpha-1}},
\end{split}
\end{equation}
for $i,j=1,2$.
Furthermore $W_{11}$ has no repeated roots and the non-abelian parabolic representations bijectively correspond to the roots of the polynomial $W_{11}\in \bz[y]$, the $Riley$ $polynomial$. Note that the degree of $W_{11}$ is $\frac{1}{2}(\alpha-1)$.

For  the case of 2-bridge link $K=S(\alpha,\beta)$,
\begin{equation}\label{link-group}
G(K)=\langle a,b \,|\,wb=bw, w^*a=aw^* \rangle,
\end{equation}
where $w$ is given as
\begin{equation}\label{presentation2}
w=a^{\epsilon_1}b^{\epsilon_2}a^{\epsilon_3}b^{\epsilon_4}\cdots b^{\epsilon_{\alpha-2}}a^{\epsilon_{\alpha-1}}
\end{equation}
and
$w^*$ as
\begin{equation}\label{presentation3}
w^*=b^{\epsilon_1}a^{\epsilon_2}b^{\epsilon_3}a^{\epsilon_4}\cdots a^{\epsilon_{\alpha-2}}b^{\epsilon_{\alpha-1}}.
\end{equation}
 Then by following Riley's argument about the knot cases, one can prove that  $W_{12}=-\frac{1}{y}W_{21}^*$, and the non-zero roots of $W_{12}$ correspond to non-abelian parabolic representations of $K$ from
$\rho(w)\rho(b)=\rho(b)\rho(w)$ and $\rho(w^*)\rho(a)=\rho(a)\rho(w^*)$. 
For this reason, we will call the polynomial $W_{12}\in \bz[y]$ the $Riley$ $polynomial$ of a link $K$ in this paper.
Note that the degree of $W_{12}\in \bz[y]$  is $\frac{1}{2}(\alpha-2)$.

We use the notation $\mathcal{R}(y)$ to denote the Riley polynomial so that $\mathcal{R}(y)=W_{11}(y)$ if $K$ is a knot and  $\mathcal{R}(y)=W_{12}(y)$ if $K$ is a link. 
See also section \ref{na-rep}.

\subsection{Chebyshev polynomials}

$Chebyshev$ $polynomials$, which are defined by a three-term recursion
$$g_{n+1}(t)=tg_n(t)-g_{n-1}(t),$$ can be used to describe some properties and characteristics of knots and links. We denote the Chebyshev polynomials $g_n(t)$ with the intial condition 
$g_0(t)=a, g_1(t)=b$ by $Ch_n^t(a,b)$, which clearly depends on the initial condition linearly.  And the following properties are also obvious:
\begin{equation*}
\begin{split}
Ch_n^t(a,b)&=Ch_{n-k}^t(Ch_k^t(a,b),Ch_{k+1}^t(a,b))\\
Ch_n^t(a,b)&=aCh_n^t(1,0)+bCh_n^t(0,1)
\end{split}
\end{equation*}
If we denote a particular Chebyshev polynomials $Ch_n^t(0,1)$ by $p_n(t)$, then 
$$Ch_n^t(1,0)=Ch_{n-1}^t(0,-1)=-Ch_{n-1}^t(0,1)=-p_{n-1}(t),$$

 and arbitrary Chebyshev polynomials are expressed as  linear combinations of $p_{n-1}(t)$ and   $p_n(t)$ as follows :
 $$Ch_n^t(a,b)=-ap_{n-1}(t)+bp_n(t)$$


Sometimes we will call such Chebyshev polynomials 
$t$-$Chebyshev$ $polynomials$ 
when it is needed to specify the variable $t$. 
The following Chebyshev polynomials 
are frequently used:
\begin{equation*}
\begin{split}
f_n(t)&:=p_{n+1}(t)-p_n(t)=Ch_{n+1}^t(1,1)=Ch_n^t(1,t-1)\\
v_n(t)&:=p_{n+1}(t)-p_{n-1}(t)=Ch_n^t(2,t)\\
\end{split}
\end{equation*}
The following properties for Chebyshev polynomials are well-known or easily proved. See \cite{Hsiao, Rivlin, Yamagishi} for references.
\begin{lemma}\label{cheby-prop}
 Let $p_n, f_n, v_n$ be as above. Then  
\begin{enumerate}
\item [\rm (i)] $p_{-n}(t)=-p_n(t),  \quad p_n(-t)=(-1)^{n+1}p_n(t)$
\item [\rm (ii)] $p_{n}(2)=n,  \quad p_{2n}(0)=0, \quad p_{2n+1}(0)=(-1)^n$
\item [\rm (iii)] $p_{2n+1}(t)=p_n(v_{2}(t))+p_{n+1}(v_{2}(t))=p_{n+1}(t)^2-p_{n}(t)^2$
\item [\rm (iv)] $p_{n}(t)^2-p_{n-1}(t)p_{n+1}(t)=1$
\item [\rm (v)] $p_{n+1}(t)^2-tp_{n}(t)p_{n+1}(t)+p_{n}(t)^2=1$
\item [\rm (vi)] $p_{nm}(t)=p_{n}(v_m(t))p_m(t)$
\item [\rm (vii)] $(-1)^nf_{n}(-t)=p_n(t)+p_{n+1}(t)$
\item [\rm (viii)] $f_{n}(-v_{2k}(t))=f_{n}(v_k(t))f_{n}(-v_k(t))$
\end{enumerate}
\end{lemma}

\begin{lemma} \label{SL-Chebyshev}
Let $M\in SL(2,\bc)$ and $t=tr M$ . Then  
\begin{enumerate}
\item [\rm (i)] $M^n=Ch_n^t(I,M)=p_n(t)M-p_{n-1}(t)I$
\item [\rm (ii)] $tr(M^n)=v_n(t)$
\end{enumerate}
\end{lemma}

\subsection{Quandle}
\begin{definition}
A set $X$ that has a binary operation $\rhd : X \times X \rightarrow X$ is called a $quandle$ if the following three axioms hold:
\begin{enumerate}
\item [\rm (i)]  for any $a \in X$, $a\rhd a=a$;
\item [\rm (ii)]  for each $a, b \in X$, there is a unique $ c \in X$ such that $c \rhd b=a$;
\item [\rm (iii)]  for each $a, b,c \in X$, $(a\rhd b) \rhd c=(a \rhd c)\rhd (b \rhd c)$.
\end{enumerate}
\end{definition}
Axiom (ii) implies that quandle operation $\rhd$ has a $dual$ quandle operation $\rhd ^{-1}$ such that $c=a \rhd ^{-1} b $. Note that  
$$(a \rhd ^{-1} b)\rhd b=a \quad \text{and} \quad (a \rhd  b)\rhd^{-1} b=a, $$ and these two operations distribute over each other. See \cite{Carter, Joyce, Yetter}. 

Any group $G$ is a quandle with respect to the operation $a\rhd b =b^n a b^{-n}, a,b\in G$  for any integer $n$.
 For any oriented link diagram $K$, there is a quandle $Q(K)$ defined by a Wirtinger-style presentation with one generator for each arc and a relation at each crossing: Let $\mathcal{R}(K)$ be the set of over-arcs of $K$ with an orientation. Then for each crossing, we have three elements $\alpha,\beta$ and $\gamma$ of $\mathcal{R}(K)$ and the knot quandle relation between them, $\gamma=\alpha \rhd \beta$, is defined as in  Figure \ref{KnotQundle}. This  quandle is called a $knot$ $quandle$ and it is known that this quandle is a classifying invariant of knots and unsplit links in $\Sphere^3$ \cite{Joyce,Nelson}. Note that quandle operation is invariant under the  Reidemeister moves by the quandle axioms.

\begin{figure}[hbt]
\begin{center}
\scalebox{0.3}{\includegraphics{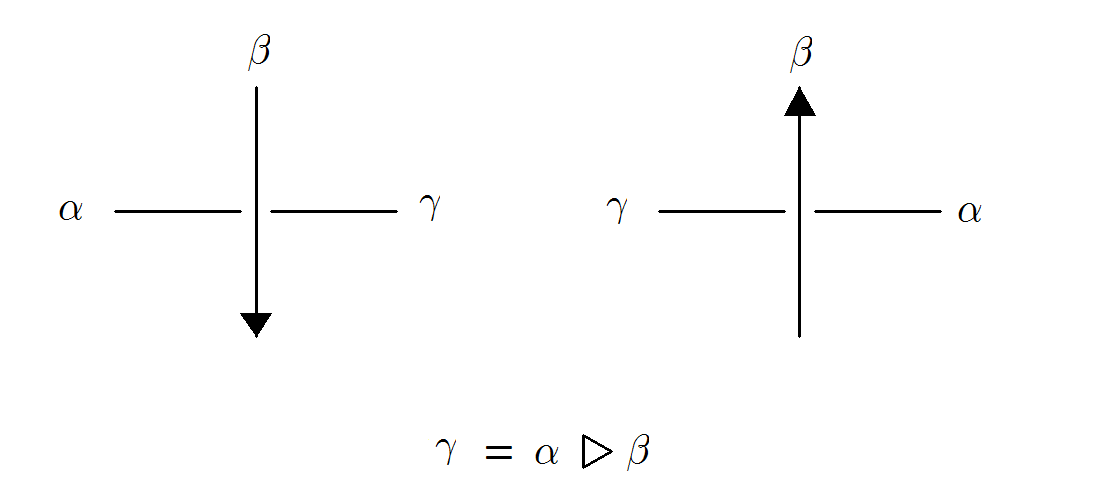}}
\end{center}
\caption{knot quandle relation}\label{KnotQundle}
\end{figure}

\begin{definition} Let $X$ be a quandle, $K$ an oriented knot or link diagram. A $quandle$ $coloring$ $\mathcal{C}$ on $K$ is a map $\mathcal{C} : \mathcal{R}(K) \longrightarrow X$ such that $$\mathcal{C}(\gamma)=\mathcal{C}(\alpha)\rhd \mathcal{C}(\beta)$$ holds  at each crossing with arcs $\alpha,\beta$ and $\gamma=\alpha \rhd \beta$. We will say that $K$ is colored by $X$ or $K$ has a $X$-coloring when there is such a quandle coloring.
\end{definition}

\subsection{Symplectic quandle}
\begin{definition} Let $X$ be a finite dimensional free module over a commutative ring with identity and a non-degenerate anti-symmetric bilinear form $\langle , \rangle$. Then  ($X$,$\langle , \rangle$) is a quandle with the quandle operation 
$$x\rhd y=x+\langle x,y \rangle y  \quad \text{and} \quad  x\rhd^{-1} y=x-\langle x,y \rangle y.$$
This type of quandle is called a $symplectic$ $quandle$. See \cite{Nelson} for more details.
\end{definition}
\section{Symplectic quandle structure on the set of parabolic elements in $SL(2,\bc)$}

\subsection{Symplectic quandle structure on $\bc^2$}
 Let $\langle , \rangle$ be a symplectic form on $\bc^2$ defined by  
$$\langle x,y \rangle=\begin{vmatrix}
x_1 & y_1\\
x_2 & y_2 
 \end{vmatrix}=x_1y_2-x_2y_1$$
for $ x=\begin{pmatrix}
x_1\\
  x_2
 \end{pmatrix},\, y=\begin{pmatrix}
y_1\\
  y_2
  \end{pmatrix}.$ 
Then $(\bc^2, \langle , \rangle)$ is a symplectic quandle with the quandle operation 
$$x\rhd y=x+\langle x,y \rangle y.$$

The symplectic quandle structure on $\bc^2$ induces a quandle structure on the space of orbits of the action of multiplicative group $\{1,-1\}$ on $\bc^2$ by scalar multilication, because negating $x$ negates $x\rhd y$ and $x\rhd^{-1} y$, while  negating $y$ leaves them unchanged. We will denote the orbit space by $\mathfrak{C}=\bc^2/{\pm 1}$ and 
call this quandle $(\mathfrak{C},\langle , \rangle)$ a ($reduced$) $symplectic$ $quandle$.


\subsection{Set of parabolic elements in $SL(2,\bc)$}
We denote the set of parabolic elements of  $SL(2,\bc)$ by $\mathcal{P}$, that is,
$\mathcal{P}=\{A\in SL(2,\bc)\mid tr(A)= 2 \}.$

From the fact that every parabolic element is conjugate to the particular element
 $\begin{pmatrix}
1 & 1\\
  0 & 1 
 \end{pmatrix}$, we get the following identities   
 \begin{equation*}
\begin{split}
A=\begin{pmatrix}
a_{11}& a_{12}\\
a_{21} & a_{22} 
 \end{pmatrix}
\begin{pmatrix}
1 & 1\\
  0 & 1 
 \end{pmatrix}
\begin{pmatrix}
a_{11}& a_{12}\\
a_{21} & a_{22} 
 \end{pmatrix}^{-1}
&=\begin{pmatrix}
1-a_{11}a_{21}& a_{11}^2\\
-a_{21}^2 & 1+a_{11}a_{21} 
 \end{pmatrix}\\
&=I+\begin{pmatrix}
a_{11}\\
a_{21}  
 \end{pmatrix}
(-a_{21},a_{11})=I+\begin{pmatrix}
-a_{11}\\
-a_{21}  
 \end{pmatrix}
(a_{21},-a_{11})
\end{split}
\end{equation*}
and 
$$A^{-1}=I-\begin{pmatrix}
a_{11}\\
a_{21}  
 \end{pmatrix}
(-a_{21},a_{11})=I-\begin{pmatrix}
-a_{11}\\
-a_{21}  
 \end{pmatrix}
(a_{21},-a_{11}).$$
This gives a bijection $T$ from  $\mathcal{P}$ to  $\mathfrak{C}$ such that 
$$T(A)=\left[\begin{array}{c}
a_{11}\\
a_{21}  
 \end{array}\right]:= [a],$$ 
and this map sends $A^{-1}\in \mathcal{P}$ to  $[ia]\in \mathfrak{C} $, that is, $T(A^{-1})=[ia]$. If we denote $(-a_{21},a_{11})$ by $\hat{a}$, then $\hat{a}b=\langle a,b\rangle$ and  $A$ and $A^{-1}$ can be expressed as
$$A=I+a\hat{a},\quad A^{-1}=I-a\hat{a}=I+(ia)(\widehat{ia}).$$

The following proposition shows that $T$ defines an isomorphism between the quandle $(\mathcal{P}, \text{conjugation})$ and the symplectic quandle $(\mathfrak{C},\langle, \rangle) $.
\begin{proposition}
	If $T(A)=[a]$ and  $T(B)=[b]$ then 
	$$T(B^{-1}AB)=[a+\langle a,b \rangle b]\in \mathfrak{C}$$
	and
	$$T(BAB^{-1})=[a-\langle a,b \rangle b]\in \mathfrak{C}.$$
\end{proposition}
\begin{proof}
	The first identity follows from
\begin{equation*}
	\begin{split}
	B^{-1}AB&=(I-b\hat{b})(I+a\hat{a})(I+b\hat{b})\\
	&=(I-b\hat{b}+a\hat{a}-b\hat{b}a\hat{a})(I+b\hat{b})\\
	&=(I-b\hat{b}+a\hat{a}-\langle b,a\rangle b\hat{a})(I+b\hat{b})\\
	&=I+a\hat{a}-\langle b,a\rangle b\hat{a}+\langle a,b\rangle a\hat{b}+\langle a,b\rangle ^2b\hat{b}\\
	&=I+a\hat{a}+\langle a,b\rangle (a\hat{b}+b\hat{a})+\langle a,b\rangle ^2b\hat{b}\\
	&=I+(a+\langle a,b\rangle b)(\hat{a}+\langle a,b\rangle \hat{b})\\
	&=I+c\hat{c}
	\end{split}
\end{equation*}	
where $c=a+\langle a,b\rangle b$ by obvious linearity of $\,\,\,\hat{ }\,\,$. The second identity is similarly proved, or it can be proved using the first identity as follows.
\begin{equation*}
	\begin{split}
		BAB^{-1}&=I+d\hat{d}, \quad  d=a+\langle a,ib\rangle ib=a-\langle a,b\rangle b\\
			\end{split}
\end{equation*}	
\end{proof}
\begin{remark}
It holds true that $a\rhd b=B^{-1}a$ for any $a,b\in \mathbb C^2$ with  $T(B)=[b]$, since
\begin{equation*}
B^{-1}a=(I-b\hat{b})a=a-b\hat{b}a=a-b\langle b,a\rangle=a+\langle a,b\rangle b=a\rhd b.
\end{equation*}
Also note that 
\begin{equation}\label{3-1}
\begin{split}
tr(AB)&=tr((I+a\hat{a})(I+b\hat{b}))\\
&=tr(I)+tr(a\hat{a})+tr(b\hat{b})+tr(a\hat{a}b\hat{b})\\
&=tr(I)+0+0+tr(\langle a,b\rangle a\hat{b})\\
&=2-\langle a,b\rangle^2  \end{split}, 
\end{equation}
and hence $\langle a,b\rangle$ tells us about $tr(AB)$.
\end{remark}
\section{Parabolic representations of 2-bridge knots and links}\label{BPrep-2brKnots}
In this section, we investigate the set of parabolic representations of a 2-bridge link $K$ using its conway normal form $J(n_1 ,\cdots,n_k)=C[n_1,-n_2, ,\cdots,(-1)^{k+1}n_k]$ and the symplectic quandle structure on $\mathcal{P}$ described in the previous section. 
Throughout the paper, we will write $K=J(n_1 ,\cdots,n_k)$ whenever we consider the link diagram $J(n_1 ,\cdots,n_k)$ for a link $K$.

Fix a Conway expansion $J(n_1 ,\cdots,n_k)$ of $K$ and its orientation.  Then each parabolic representation corresponds to a $\mathfrak{C}$-coloring  and vice versa. To get a $\mathfrak{C}$-coloring, we start from two vectors $a_{1,0}, b_{1,0}$ of  $\bc^2$ and obtain two $(i,j)$-vectors, $a_{i,j},b_{i,j}$ for all $i=1,\cdots,  k$ and 
$$j=0,1,2,\cdots, n_i \ \text{if} \ n_i>0 \quad \text{and} \quad
j=0,-1,-2,\cdots, n_i \ \text{if} \ n_i<0,$$
 which are the  $|j|$-th vectors of  the $i$-th block in the diagram of $K$, consecutively obtained  by performing the quandle operation of $\mathfrak{C}$ at every crossing while descending down. The last two vectors of each $i$-th block, $a_{i,n_i},b_{i,n_i}$, will be denoted by $ a_{i,f},b_{i,f}$ sometimes for our convenience. (See Figure \ref{KeyLemma}.)
 
At each crossing we will take ``$-$" sign for a new vector, that is, 
\begin{equation}\label{quandle-eq1}
a'=b, \quad b'   = -(a\rhd b)= -a-\langle a,b\rangle b
\end{equation} 
when the crossing is given as in the left-hand side of Figure \ref{QuandleAction}, and 
\begin{equation}\label{quandle-eq2}
a'= b, \quad b'=-(a\rhd^{-1} b)=-a +\langle a,b \rangle b
\end{equation}
when the crossing is given as in the right-hand side  of Figure \ref{QuandleAction}. 
\begin{figure}[hbt]
\begin{center}
\scalebox{0.25}{\includegraphics{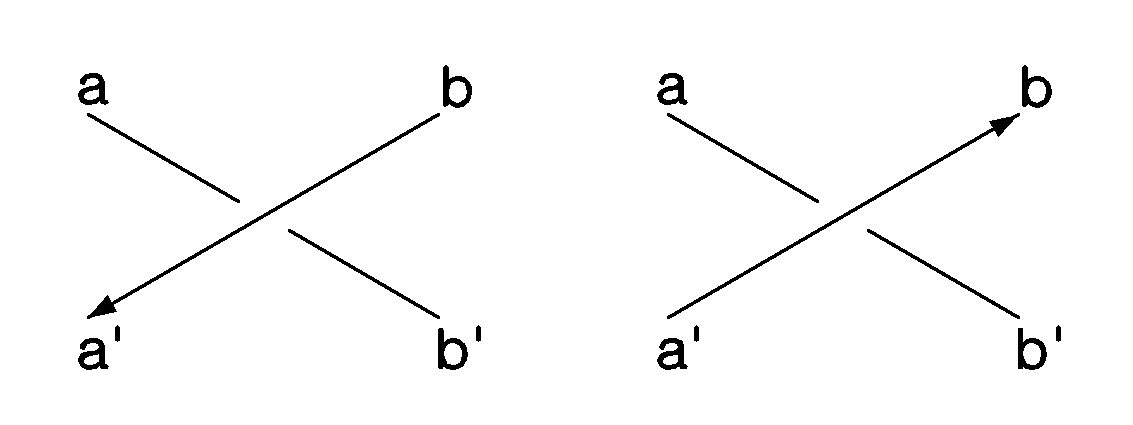}}
\end{center}
\caption{  }\label{QuandleAction}
\end{figure}
Notice that our choice of ``$-$" sign is to have a consistent determinant 
 $\langle a',b' \rangle=\langle a,b \rangle$ for each block and thus we get
\begin{equation}\label{quandle-block}
u_i=\langle a_{i,0},b_{i,0}\rangle=\cdots=\langle a_{i,j},b_{i,j}\rangle=\cdots=\langle a_{i,f},b_{i,f}\rangle
\end{equation}
for each $i$, and $u_i$ will be called the {\it determinant of $i$-th block}.

If we let $\langle a,b \rangle =u$ and $X(u)=\begin{pmatrix}
0 & -1\\
1 & -u 
\end{pmatrix}$, then Equation (\ref{quandle-eq1}) and  Equation (\ref{quandle-eq2}) can be expressed as 
\begin{equation*}
	(a',b')=(a,b)\begin{pmatrix}
		0 & -1\\
		1 & -u 
	\end{pmatrix}=(a,b)X(u) \quad \text{and}\quad
	(a',b')=(a,b)\begin{pmatrix}
		0 & -1\\
		1 & u 
	\end{pmatrix}=(a,b)X(-u),
\end{equation*}
respectively. Therefore  we get from (\ref{quandle-block}) that for each $i$
\begin{equation*}
	(a_{i,j+\delta_i'},b_{i,j+\delta_i'})=(a_{i,j}, b_{i,j})X(\delta_i u_i)^{\delta_i'},
\end{equation*}
where 
$\delta_i=1$ if the orientation of the crossing is downward and  $\delta_i=-1$ if the orientation of the crossing is upward, and $\delta_i'=1$ if $n_i>0$  and  $\delta_i'=-1$ if $n_i<0$ (see Figure \ref{Q-Actions}).
\begin{figure}[hbt]
	\begin{center}
		\scalebox{0.6}{\includegraphics{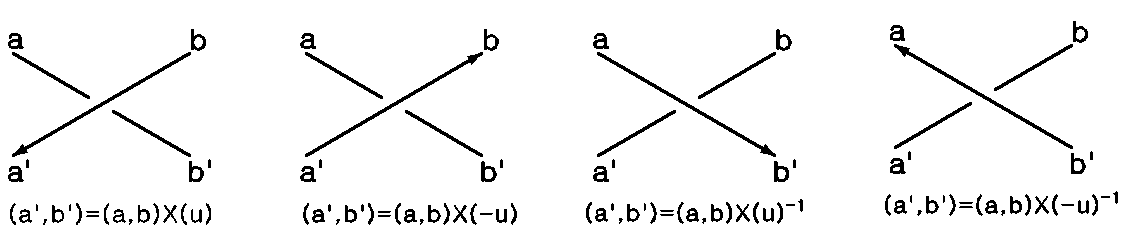}}
	\end{center}
	\caption{quandle action when $u=\langle a,b \rangle$}\label{Q-Actions}
\end{figure}

Since the presentation of the knot group (or link group) of $K$ is generated by two  meridians which are conjugate each other,  
and the meridians correspond to initial two vectors $a_{1,0}$ and $b_{1,0}$, there are 3 cases. The first case is that the representation is trivial, that is, in this case $$a_{1,0}=\begin{pmatrix}
0\\
0  
 \end{pmatrix}=b_{1,0},$$ 
which corresponds to  
$\begin{pmatrix}
1& 0\\
0 & 1 
 \end{pmatrix} \in SL(2,\bc)$.
The second case is that  the representation is a non-trivial abelian representation. In this case,  we can normalize the meridians up to conjugate so that 
 $$a_{1,0}=\begin{pmatrix}
1\\
0  
 \end{pmatrix}, b_{1,0}=\begin{pmatrix}
v\\
0  
 \end{pmatrix},$$ 
which corresponds respectively  to  
$\begin{pmatrix}
1& 1\\
0 & 1 
 \end{pmatrix}$ and 
$\begin{pmatrix}
1& v^2\\
0 & 1 
 \end{pmatrix}
$ in $SL(2,\bc)$. 
The last case is  that the representation is a non-abelian representation. In this case, we can normalize the meridians up to conjugate so that 
\begin{equation}\label{normalize-nonabelian}
a_{1,0}=\begin{pmatrix}
1\\
0  
 \end{pmatrix}, b_{1,0}=\begin{pmatrix}
0\\
u  
 \end{pmatrix},  u\neq 0,
 \end{equation}
 which corresponds respectively to  
$\begin{pmatrix}
1& 1\\
0 & 1 
 \end{pmatrix}$ and 
$\begin{pmatrix}
1& 0\\
-u^2& 1
 \end{pmatrix}$  
 in $SL(2,\bc)$. 
Here $u^2=y$ in (\ref{parabolic-rep}).

Note that the first and second cases are when $u=u_1=\langle a_{1,0},b_{1,0}\rangle=0$
and the third case is when $u=u_1=\langle a_{1,0},b_{1,0}\rangle \neq 0$. Note that the second case with $v\neq 1$ is possible  only for links. 
 Also, there is a parabolic representation  
 $\rho : G(K)  \rightarrow SL(2,\mathbb C)$ with $tr(\rho (ab))=2-u^2$
  if there is  a  $\mathfrak{C}$-coloring on $K$ with $\langle T(\rho(a)),T(\rho(b))\rangle=\pm u$ from (\ref{3-1}). 

\subsection{Key lemmas} 
In this section, we assume that the orientation of $K=J(n_1,\cdots,n_k)$ is given.
\begin{figure}[hbt]
\begin{center}
\scalebox{0.7}{\includegraphics{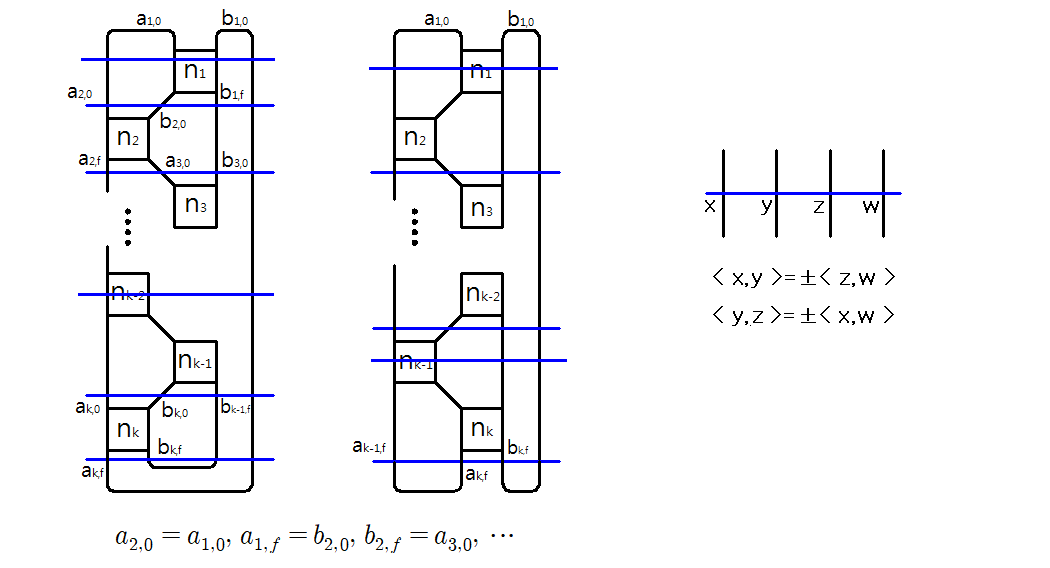}}
\end{center}
\caption{  }\label{KeyLemma}
\end{figure}
\begin{lemma}\label{KeyL}Let $x,y,z,w$ are vectors in $\mathbb C^2$ which sequentially correspond to the 4 arc vectors intersecting an arbitrary horizontal line.(See Figure \ref{KeyLemma}.) Then 
\begin{equation}\label{K-L-eq}
\langle x,y\rangle=\pm\langle z,w\rangle \quad \text{and} \quad \langle x,w\rangle=\pm\langle y,z\rangle.
\end{equation}
\end{lemma}
\begin{proof}
Since the element in $SL(2,\bc)$
corresponding to the loop rotating our diagram horizontally by 1 full turn  is the identity matix, 
one of the following is satisfied:
\begin{enumerate}
	\item [\rm (i)] $XYZ^{-1}W^{-1}=Id$
	\item [\rm (ii)] $X^{-1}Y^{-1}ZW=Id$
	\item [\rm (iii)] $X^{-1}YZW^{-1}=Id$
	\item [\rm (iv)] $X^{-1}YZ^{-1}W=Id$
	\item [\rm (v)] $XY^{-1}ZW^{-1}=Id$
	\item [\rm (vi)] $XY^{-1}Z^{-1}W=Id$
\end{enumerate}	
where $X,Y,Z,W$ are the elements in $SL(2,\mathbb C)$ which correspond to $x,y,z,w$, respectively.
If (i) is satisfied, then 
 $$\langle x,y\rangle^2=2-tr(XY)=2-tr(WZ) = \langle z,w\rangle^2$$
and    
\begin{equation*}
	\begin{split}
		\langle x,w\rangle^2&=-\langle ix,w\rangle^2=-(2-tr(X^{-1}W))=-(2-tr(YZ^{-1}))=-\langle y,iz\rangle^2= \langle y,z\rangle^2.
	\end{split}
\end{equation*}	
In the case that any of (ii)-(vi) is satisfied, we also get the same result, 
 $$\langle x,y\rangle^2=\langle z,w\rangle^2, \quad \langle x,w\rangle^2=\langle y,z\rangle^2,$$ by a similar argument.
This completes the proof.
\end{proof}
From now on, we will use $a$ and $b$   instead of  $a_{1,0}$ and $b_{1,0}$ for simplicity if there is no worry about confusion.
\begin{corollary}\label{KL-coro}
\begin{enumerate}
\item [\rm (i)] $\langle b,a_{2j,f}\rangle=\pm u_{2j+1}$ if $2j+1 \leq k$.
\item [\rm (ii)] $\langle b,b_{2j+1,f}\rangle=\pm u_{2j+2}$  if $2j+2 \leq k$.
\item [\rm (iii)] $\langle a_{i,0},a_{i,f}\rangle =\pm \langle b_{i,0},b_{i,f} \rangle$ for each $i=1,\cdots,k$.
\end{enumerate}
\end{corollary}
\begin{proof}
\begin{enumerate}
\item [\rm (i)] By Lemma \ref{KeyL}, $\langle b,a_{2j,f}\rangle=\pm \langle a_{2j+1,0},b_{2j+1,0}\rangle=\pm u_{2j+1}$.
\item [\rm (ii)] By Lemma \ref{KeyL}, $\langle b,b_{2j+1,f}\rangle=\pm \langle a_{2j+2,0},b_{2j+2,0}\rangle=\pm u_{2j+2}$.
\item [\rm (iii)] 
 The  $\mathfrak{C}$-coloring of the $i$-th block of $J(n_1 ,\cdots,n_k)$ starting from two vectors $a_{1,0},b_{1,0}$ with $u=u_1=\langle a_{1,0},b_{1,0}\rangle$  is the same as the $\mathfrak{C}$-coloring of the first block of $K'$, which is one of  $$J(n_i),\,\, J(n_i,2)$$
 with the same orientation as $K$ or  reversed, starting from two vectors $a'_{1,0}=a_{i,0}$ and $b'_{1,0}=b_{i,0}$.  If we apply Lemma \ref{KeyL} to $K'$, then  
$$\langle a'_{1,0},a'_{1,n_i}\rangle=\pm \langle b'_{1,0},b'_{1,n_i}\rangle,$$
which implies 
$$\langle a_{i,0},a_{i,f}\rangle=\pm \langle b_{i,0},b_{i,f} \rangle.$$
(iii) can be also proved directly by the same argument as Lemma \ref{KeyL}, because $$A_{i,0}^{\epsilon_1}A_{i,f}^{\epsilon_2}B_{i,f}^{\epsilon_3}B_{i,0}^{\epsilon_4}=Id$$
for some $\epsilon_1, \epsilon_2,\epsilon_3,\epsilon_4 \in \{1,-1\}$ with $\epsilon_1+\epsilon_2+\epsilon_3+\epsilon_4=0$.
\end{enumerate}
\end{proof}

The construction of $a_{i,j}$ and $b_{i,j}$ implies that there are  polynomials $f_{i,j}(u), g_{i,j}(u), \tilde{f}_{i,j}(u), \tilde{g}_{i,j}(u)$ for each $i=1,\cdots,k$ and $j=0,\cdots,n_i$ such that 
$$
a_{i,j}=f_{i,j}(u)a+g_{i,j}(u) b,  \quad b_{i,j}=\tilde{f}_{i,j}(u)a+\tilde{g}_{i,j}(u) b,
$$
and these polynomials have the following properties.

\begin{lemma}\label{degree} 
\begin{enumerate}
\item [\rm (i)] $f_{i,j}(u), g_{i,j}(u), \tilde{f}_{i,j}(u), \tilde{g}_{i,j}(u)$ are monic polynomials with integer coefficients for any pair $(i,j)$ up to sign.
\item [\rm (ii)] $u_i$ is also a monic integer coefficient polynomial of $u$ up to sign, and $u \,|\, u_i$ for all $i=1,\cdots,k$.
\end{enumerate}
\end{lemma}
\begin{proof} 
For $i=1$, (i) is obvious from the definition of $f_{i,j}(u), g_{i,j}(u), \tilde{f}_{i,j}(u), \tilde{g}_{i,j}(u)$ and (ii) is trivially satisfied because $u_1=u$. Since $u_2=\langle a_{2,0},b_{2,0}\rangle=\langle a_{1,0},a_{1,f}\rangle=ug_{1, f}(u)$, $u_2$ is a monic integer coefficient polynomial of $u$ and $u \,|\, u_2$  if $u_2\neq 0$, and obviously 
$f_{2,j}(u), g_{2,j}(u), \tilde{f}_{2,j}(u), \tilde{g}_{2,j}(u)$ are all monic polynomials with integer coefficients. 
Note thtat 
$$\deg f_{i,j+1}=\deg f_{i,j}+\deg u_i$$
for any pair $(i,j)$, and the same is also true for $g_{i,j}(u), \tilde{f}_{i,j}(u), \tilde{g}_{i,j}(u)$.

Now we proceed by induction on $i$. Assume that the statement (i) and (ii) are true for all $i\leq k$.
Since 
\begin{equation*}
u_{k+1}=\pm \langle b,a_{k,f}\rangle=\pm uf_{k,f} \quad \text{when}\,\, k\,\, \text{is even }
\end{equation*}
and 
\begin{equation*}
u_{k+1}=\pm \langle b,b_{k,f}\rangle=\pm u\tilde{f}_{k, f} \quad \text{when}\,\, k\,\, \text{is odd }
\end{equation*}
by Corollary \ref{KL-coro}, $u_{k+1}$ is also monic and is divided by $u $. Therefore we can conclude  that (i) and (ii) are true for $i=k+1$, which completes the proof. 
\end{proof}

\begin{lemma}\label{K-L2} Suppose $u=u_1\neq 0$. Let $x,y,z,w$ be vectors in $\mathbb C^2$ which sequentially correspond to the 4 arc vectors intersecting an arbitrary horizontal line (therefore $w=b_{1,0}$). Then the following holds  for some $\epsilon\in \{1,i\}$.
\begin{enumerate}
\item [\rm (i)] If $\langle x,y\rangle=0$ or $\langle z,w\rangle=0$, then $x=\pm\epsilon y$ and  $z=\pm\epsilon w$.
\item [\rm (ii)] If $\langle y,z\rangle=0$ or $\langle x,w\rangle=0$, then $y=\pm\epsilon z$ and  $x=\pm\epsilon w$.
\item [\rm (iii)] If $\langle x,z\rangle=0$, then $x=\pm\epsilon z$ and $w=\pm \epsilon( y \rhd^{\pm 1} z)$.
\item [\rm (iv)] If $\langle y,w\rangle=0$, then $y=\pm\epsilon w$  and $z=\pm \epsilon( x \rhd^{\pm 1} y)$.
\end{enumerate}
Here $\epsilon=i$ when the orientations are  opposite, and  the sign of $\rhd^{\pm 1}$ in \emph{(iii)} (respectively, \emph{(iv)}) is $+$ only when the orientation of $z$ (respectively, $y$) is going down.  
\end{lemma}
\begin{proof}
Note that by Lemma \ref{KeyL}, $\langle x,y\rangle=0$ if and only $\langle z,w\rangle=0$, and $\langle x,w\rangle=0$ if and only $\langle y,z\rangle=0$.

Firstly, we prove the lemma for the case that the horizontal line cuts the first block, that is, 
\begin{equation}\label{first block}
	x=a,y=a_{1,j}, z=b_{1,j}, w=b.
\end{equation}
	 Since an arc vector is multiplied by $i$ if its orientation is reversed, we may assume that the orientations of both of the two strands are going down as in Figure \ref{K-L-3}. The arc vectors of the diagram in Figure \ref{K-L-3} satisfy the following identity by Lemma \ref{SL-Chebyshev}, 
\begin{equation*}
\begin{split}
(a_{1,k},b_{1,k})&=(a,b)
X(u)^{k}=(a,b)
\begin{pmatrix}
-p_{k-1}(-u)& -p_{k}(-u)\\
p_{k}(-u) &p_{k+1}(-u)
 \end{pmatrix}.
\end{split}
\end{equation*}
(The left diagram is when $n_1>0$ and the right one is when $n_1<0$.)

If  $\langle a,a_{1,j} \rangle=0$ then $p_{j}(-u)=0$ and so $p_{j-1}(-u)^2=p_{j+1}(-u)^2=1$ by (v) of Lemma \ref{cheby-prop}, which implies 
$$a_{1,j}=\pm a, \quad b_{1,j}=\pm b.$$ 
This proves that  (i) is true for any $j$.
\begin{figure}[hbt]
\begin{center}
\scalebox{0.25}{\includegraphics{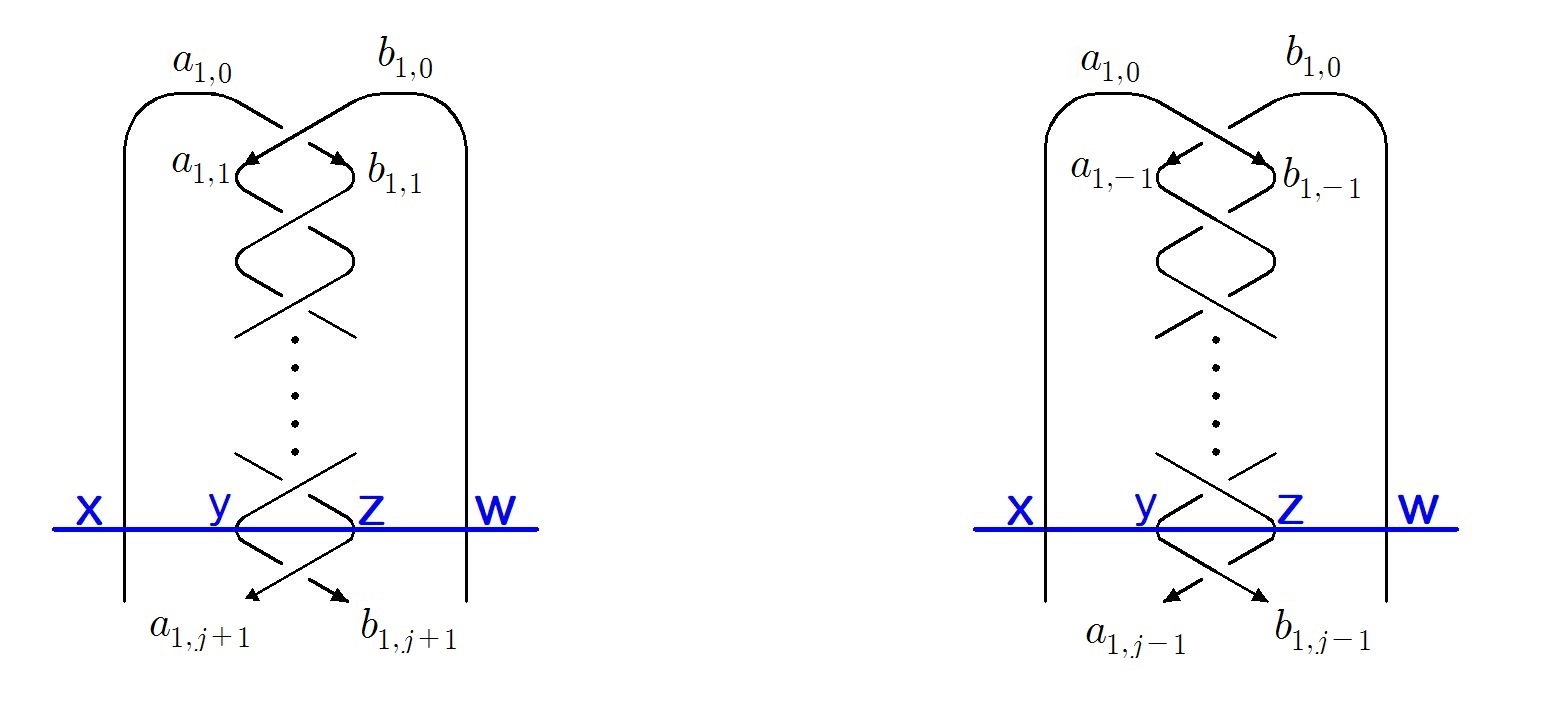}}
\end{center}
\caption{  }\label{K-L-3}
\end{figure}

If  $\langle  a,b_{1,j} \rangle=0$, then $\langle  a,a_{1,j+1} \rangle=0$  and thus 
$$b_{1,j}=a_{1,j+1}=\pm a, \quad a_{1,j}\rhd b_{1,j}= b_{1,j+1}=\pm b$$ 
holds true by (i), which proves (iii). 
If  $\langle  a_{1,j},b \rangle=0$ then $$\langle  b_{1,j-1},b \rangle=0$$ is satisfied 
and thus 
$$a_{1,j}=b_{1,j-1}=\pm b, \quad b_{1,j}\rhd^{-1} a_{1,j}= a_{1,j-1}=\pm a$$ 
 holds true 
by (i), which proves (iv).  Since $\langle  y,z \rangle=\langle  a_{1,j},b_{1,j} \rangle=u\neq 0$, we have completed the proof for the case  when the horizontal line cuts the first block. 

\begin{figure}[hbt]
\begin{center}
\scalebox{0.4}{\includegraphics{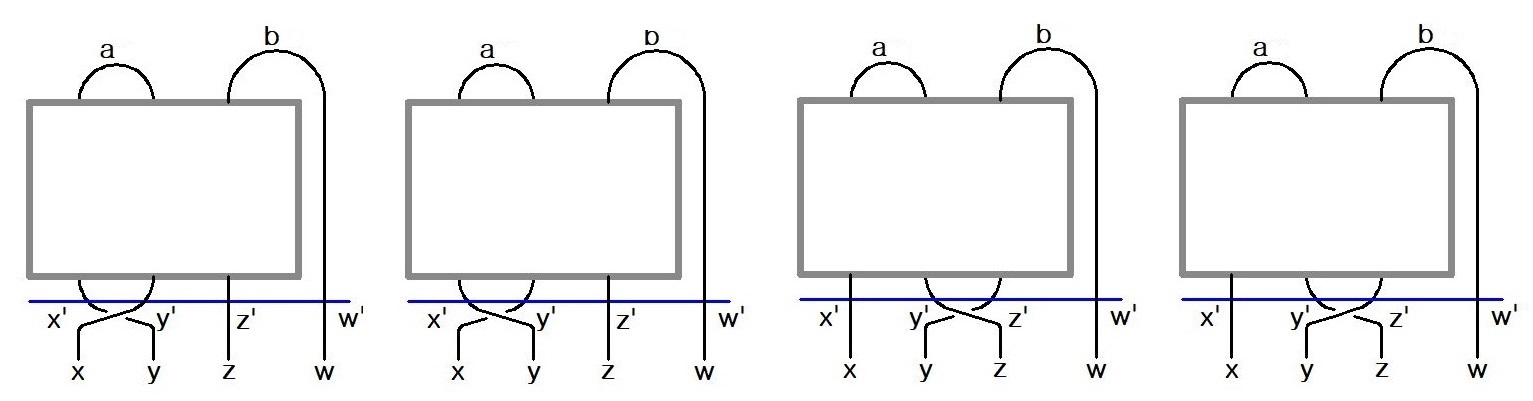}}
\end{center}
\caption{  }\label{induction}
\end{figure}

Now we proceed by induction on $(i,j)$. Assume that the lemma is true for any four $(i,j)$  pairs such that $i<k$ or $i= k, |j|<m$ and consider  the next four  $(i,j)$ pairs of arc vectors 
$x,y,z,w$, that is to say  $(i,|j|)=(k,m)$ if $m < |n_k|$ and $(i,j)=(k+1,0)$ if $m -1= |n_k|$. If we let $x',y',z',w'$ be the 4 arc vectors on the previous horizontal line, then it can be proved that these four vectors satisfy one of the assumptions of (i), (ii), (iii), and (iv) if $x,y,z,w$ do. For example, in the case of  $\langle x,y\rangle=0$, one of three equations, $$\langle x',y'\rangle=0, \langle x',z'\rangle=0,  \langle y',w'\rangle=0,$$ is satisfied by Lemma \ref{KeyL}:  $\langle x',y'\rangle=0$ holds when $z'=z$ and either $\langle x',z'\rangle=0$ or $\langle y',w'\rangle=0$ holds when $x=x'$ (see Figure \ref{induction}). Since the lemma is true for $x',y',z',w'$ by the induction hypothesis, it is easy to show that $x=\pm\epsilon y$ and  $z=\pm\epsilon w$, which means (i) is true for $x,y,z,w$. (ii), (iii), and (iv) are similarly proved.
\end{proof}

\begin{remark} 
If we let  $x,y,z,w$ be vectors in Lemma \ref{K-L2} and $K'$ be the knot or the link which is made by closing the 4 arcs of the upper tangle, as in the diagrams of Figure \ref{closing}. Then 
 the arc-coloring  corresponds to  a parabolic representation  
 $\rho : G(K') \rightarrow SL(2,\bc)$ and the orientation of $K'$ is the same as $K$  if $\epsilon=1$, and one of the two strand's orientations must be reversed  if $\epsilon=i$.
\begin{figure}[hbt]
\begin{center}
\scalebox{0.7}{\includegraphics{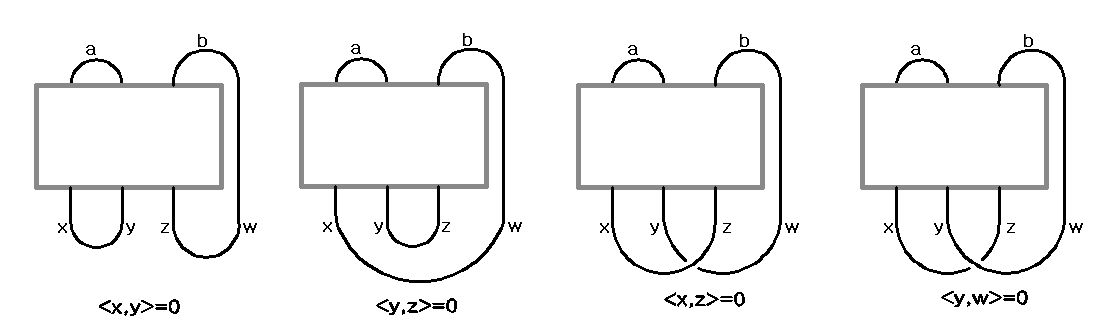}}
\end{center}
\caption{  }\label{closing}
\end{figure}
\end{remark}

\subsection{Rep-polynomials}  
To get a well-defined $\mathfrak{C}$-coloring on $K=J(n_1 ,\cdots,n_k)$ with the first two vectors $a=a_{1,0}$ and $b=b_{1,0}$,  
the last two vectors $a_{k,f}=a_{k,n_k}$ and $b_{k,f}=b_{k,n_k}$ must be the same, up to sign, as the vectors aleady determined for the arcs, which gives us the equation to determine the coloring (see Figure \ref{u-poly}).

\begin{figure}[hbt]
\begin{center}
\scalebox{0.3}{\includegraphics{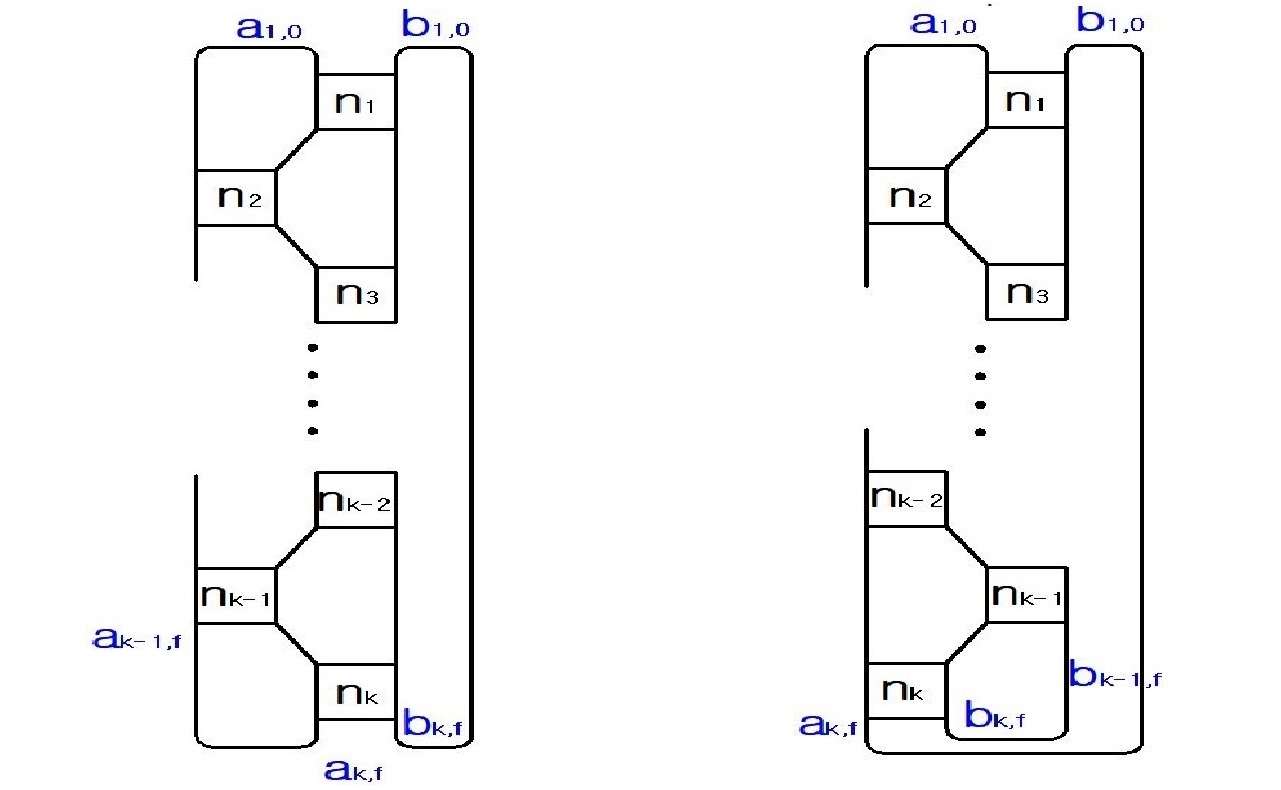}}
\end{center}
\caption{  }\label{u-poly}
\end{figure}

By Lemma \ref{KeyL},
  $$\langle a_{k,f},a_{k-1,f}\rangle=\pm\langle b_{k,f},b\rangle$$ 
when $k$ is an odd number greater than $1$, and
 $$\langle b_{k,f},b_{k-1,f}\rangle=\pm\langle a_{k,f},b\rangle $$ 
when $k$ is an even number.
If we let $a_{0,f}=a$, then the above equation also holds for the case when  $k=1$. 

Let $P_K(u)$ be  a polynomial in $u$ with positive leading coefficient which is defined as follows:
\begin{enumerate}
\item [\rm (i)]    $P_K(u)=\pm \langle a_{k,f},a_{k-1,f}\rangle=\pm \langle b_{k,f},b\rangle $ if $k$ is odd,
\item [\rm (ii)] $P_K(u)=\pm \langle b_{k,f},b_{k-1,f}\rangle=\pm \langle a_{k,f},b\rangle $ if $k$ is even. 
\end{enumerate}
Note that $P_K(u)$ is defined for a diagram $K=J(n_1,\cdots,n_k)$, but we will see in Theorem \ref{2-rep-poly} that it is essentially independent of the choice of a diagram.
\begin{proposition} \label{prop-uPolynomial} 
Let  $K=J(n_1,\cdots,n_k)$ with an orientation.
Then $ P_K(u)$ is a monic polynomial  with integer coefficients and there is a $\mathfrak{C}$-coloring with $\langle a,b\rangle=r$ on $K=J(n_1,\cdots,n_k)$ if and only if $r$ is a root of 
the equation $P_K(u)=0$. Furthermore, $0$ is a root of $P_K(u)$, which corresponds to an abelian representation, and  there is a $\mathfrak{C}$-coloring with $\langle a,b\rangle=0, a\neq \pm b$ if and only if $K$ is a link.
\end{proposition}
\begin{proof}
We can choose an integer $n_{k+1}$ such that the orientation of $K$ is the same as $K'=J(n_1,\cdots,n_k, n_{k+1})$. Then  
$P_K(u)$ equals $\pm u_{k+1}$ of $K'$ by Corollary \ref{KL-coro} and thus   $P_K(u)$ is a monic polynomial  with integer coefficients  by Lemma \ref{degree}. 

If we start with $a=b$, then $\{a_{i,j}, b_{i,j}\}\subset\{b,\ -b\}$ for all $i,j$, which implies  $P_K(0)=0$ for any $K$. It is obvious that  if $K$ is a link then any pair $ a,b$ such that $\langle  a,b \rangle=0$ always gives a  $\mathfrak{C}$-coloring on $K$, but if $K$ is a knot then  a $\mathfrak{C}$-coloring is obtained only when $a=\pm b$. 

In the case of  $\langle a,b\rangle\neq 0$, 
the followingh must be satisfied   :
 \begin{equation}\label{eq1}\{\begin{array}{rcl} a_{k,f}=&\pm a_{k-1,f}\\
b_{k,f}=&\pm b
\end{array} \mbox{for odd} \,\,k
\end{equation}
\begin{equation}\label{eq2}
\{\begin{array}{rcl} a_{k,f}=&\pm b\\
b_{k,f}=&\pm b_{k-1,f}
\end{array} \mbox{for even} \,\, k 
\end{equation}
But by Lemma \ref{K-L2},
(\ref{eq1}) is equivalent to $$ \langle a_{k,f},a_{k-1,f}\rangle=0 \quad  (\Leftrightarrow  \langle b_{k,f},b\rangle=0),
$$ 
and
(\ref{eq2})
is equivalent to 
 $$\langle a_{k,f}, b\rangle=0 \quad  (\Leftrightarrow  
\langle b_{k,f},b_{k-1,f}\rangle=0).$$

\end{proof}
Note that (i) $P_K(r)=0$ implies $P_K(-r)=0$, because if we get a $\mathfrak{C}$-coloring on $K$ from a pair $a,b$, then we must also get a $\mathfrak{C}$-coloring on $K$ from a pair $a,-b$ since it only changes the sign of the coloring from $a,b$,
and  (ii) each root $r$ of the equation $P_K(u)=0$ gives a parabolic representation of $K=J(n_1,\cdots,n_k)$ and there is no other parabolic representations. So 
we will call the polynomial $P_K(u)$, the $rep$-$polynomial$ of a 2-bridge link $K$. Even though we defined the rep-polynomial when we have an orientation, but it does not depend on its orientation for the knot case and we have two rep-polynomials for the link case as we can see in the following Proposition.
\begin{proposition}\label{u-polys}
Let $K=C[n_1,\cdots,n_k]$. If  $-K$  and $\bar{K}$ are the orientation-reversed link of $K$ and the mirror of $K$ respectively, then 
$$P_{-K}(u)=P_K(u)=P_{\bar{K}}(u).  $$
Especially, the rep-polynomial of $C[n_1,\cdots,n_k]$ equals the rep-polynomial of $C[-n_1,\cdots,-n_k]$.
\end{proposition}
\begin{proof}
Since there is a 1-1 correspondence between the set of  $\mathfrak{C}$-colorings on  $K$ and that of $-K$ by multiplying $i$ to each corresponding arc vector, and 
  $$\langle  ia, ib\rangle=-\langle  a,b\rangle,$$ $P_{-K}(u)=P_K(u)$ follows.

If we reflect an oriented diagram of $K=C[n_1,\cdots,n_k]$ in the mirror and then reverse its orientation, then we get an oriented  diagram $\bar{K}$.  It is easy to check that we get an well-defined $\mathfrak C$-coloring on $\bar{K}$ by reflecting any $\mathfrak C$-coloring on $K$. (See Figure \ref{mirror-diagram}.) 
\end{proof}
\begin{figure}[hbt]
\begin{center}
\scalebox{0.4}{\includegraphics{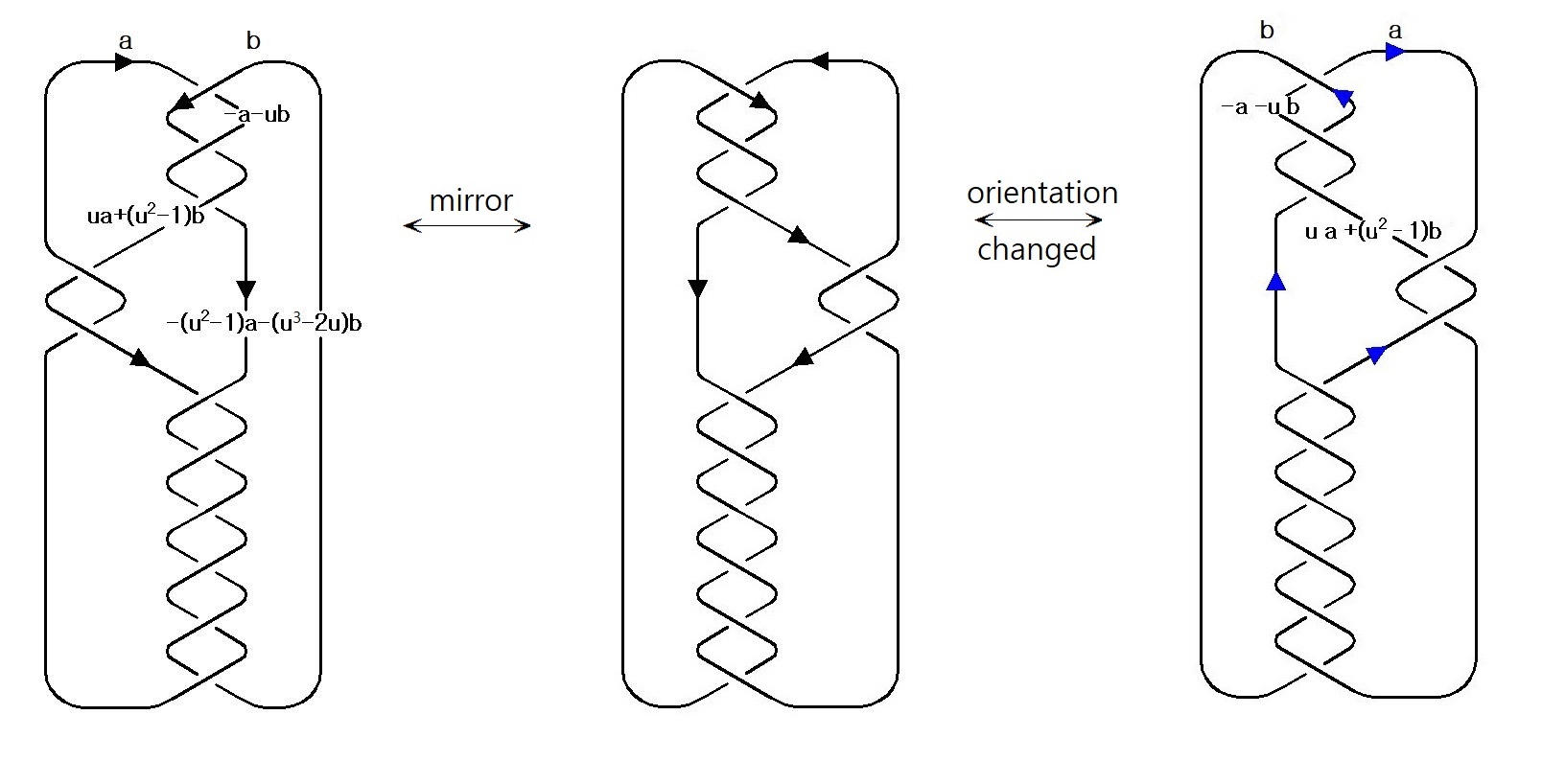}}
\end{center}
\caption{}\label{mirror-diagram}
\end{figure}
\begin{example}\label{C[3]-knot}
 The last two vectors of  the trefoil $K=3_1=C[3]$  are 
$$a_{1,3}=ua+(u^2-1)b, \quad  b_{1,3}=-(u^2-1)a-(u^3-2u)b,$$
if we start with two vectors $a_{1,0}=a, b_{1,0}=b$ such that $u=\langle a,b \rangle$. (See Figure \ref{trefoil}.) Therefore the rep-polynomial of $K$ is $$\langle b, b_{1,3} \rangle=u(u^2-1).$$
\end{example}
\begin{figure}[hbt]
\begin{center}
\scalebox{0.3}{\includegraphics{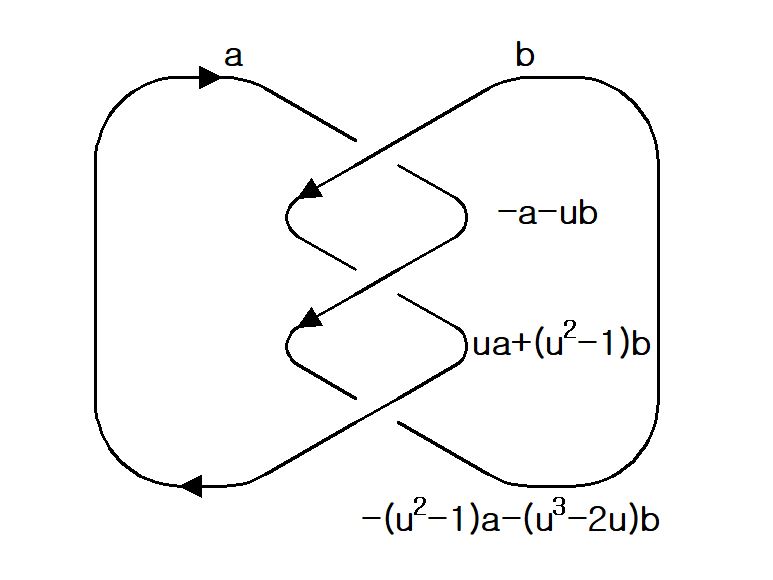}}
\end{center}
\caption{C[3]}\label{trefoil}
\end{figure}
By Proposition \ref{u-polys},  we can define the rep-polynomial of $K$ without any specific orientation for a 2-bridge knot  $K=C[n_1,\cdots,n_k]$. But if   $K=C[n_1,\cdots,n_k]$ is a link, then we get a different rep-polynomial  when we change the orientation of only one of the two components. So each link $K$ has two rep-polynomials, $P_1(u)$ and $P_2(u)$, up to its orientation, and these two satisfy  $$P_1(iu)=\pm P_2(u),\quad P_2(iu)=\pm P_1(u),$$  since 
 $$\langle  ia,b\rangle= i\langle  a,b\rangle=\langle  a, ib\rangle.$$

Our definition of the rep-polynomial of a 2-bridge link depends on its diagram, but $P_{C[n_1,\cdots,n_k]}(u)=P_{C[m_1,\cdots,m_l]}(u)$ if $[n_1,\cdots,n_k]=[m_1,\cdots,m_l]$, since $C[n_1,\cdots,n_k]$ can be deformed into $C[m_1,\cdots,m_l]$  by a finite number of Reidemeister moves.
Hence we have
 \begin{theorem}\label{2-rep-poly}
 Let $K$ be a 2-bridge knot. Then there are only two rep-polynomials $P_K(u)$ and $P'_K(u)$ for any Conway expansion diagram of $K$
 and these two satisfy the followings.
 \begin{enumerate}
\item [\rm (i)] two diagrams $C[n_1,\cdots,n_k]$ and $C[m_1,\cdots,m_{k'}]$ of $K$ have the same rep-polynomials if $[n_1,\cdots,n_k]=\pm [m_1,\cdots,m_{k'}]$, 
\item [\rm (ii)] if $P_K(u)$ is the rep-polynomial of a diagram $C[n_1,\cdots,n_k]$, then $P'_K(u)$  is the rep-polynomial of the upside-down diagram $C[(-1)^{k+1}n_k,\cdots,(-1)^{k+1}n_1]$.
\end{enumerate}
Similarly, each 2-bridge link $K$ has  four rep-polynomials $P_K(u), P_K(iu), P'_K(u), P'_K(iu)$.
 \end{theorem}

Each rational number $\frac{\beta}{\alpha}\in (0,1)$  corresponds to a 2-bridge link $C[n_1,\cdots,n_k]$ with $[n_1,\cdots,n_k]=\frac{\beta}{\alpha}$. 
So by Theorem \ref{2-rep-poly} we have two polynomials  
$$P_{\frac{\beta}{\alpha}}(u), \,\, P'_{\frac{\beta}{\alpha}}(u) \in \mathbb Z[u]$$ if  we give the downward-orientation on both components for the  case when $C[n_1,\cdots,n_k]$ is a link. These polynomials satisfy the followings:
\begin{enumerate}
\item [\rm (i)] $P_{\frac{\beta}{\alpha}}(u)=P_{C[n_1,\cdots,n_k]}(u)$ if $[n_1,\cdots,n_k]=\frac{\beta}{\alpha}$ 
\item [\rm (ii)] $P_{\frac{\beta}{\alpha}}(u)=P_{\frac{\beta'}{\alpha'}}(u)$ if and only if $\frac{\beta}{\alpha}=\frac{\beta'}{\alpha'}$ or $\frac{\alpha-\beta}{\alpha}=\frac{\beta'}{\alpha'}$
\item [\rm (iii)] $P_{\frac{\beta'}{\alpha'}}(u)=P'_{\frac{\beta}{\alpha}}(u)$, if $\frac{\beta}{\alpha}$ and $\frac{\beta'}{\alpha'}$  represent the same links and $\frac{\beta}{\alpha}\neq \frac{\beta'}{\alpha'}, \frac{\alpha-\beta}{\alpha}\neq \frac{\beta'}{\alpha'}$. In this case,  $\alpha=\alpha',  \ \beta\beta'\equiv \pm 1$ $(\text{mod} \,\alpha)$ if we assume that $(\alpha,\beta)=(\alpha',\beta')=1$.
\end{enumerate}

Note that we will see later that   if $(\alpha,\beta)=1$, then   $\deg P_{\frac{\beta}{\alpha}}(u)=\alpha$ and
\begin{equation}\label{Riley-uPoly}
\frac{1}{u^{\epsilon}}P_{\frac{\beta}{\alpha}}(u)=\pm\mathcal{R}(u^2),\end{equation}
where $\mathcal{R}(y)$ is the Riley polynomial of $S(\alpha,\beta)$ and  $\epsilon=1$ if $\alpha$ is odd, $\epsilon=2$ if $\alpha$ is even.

\subsection{$u_i$-sequence}

\begin{proposition}\label{degree2}
Let $K=C[n_1,\cdots,n_k]$  and $\alpha_i$ be an integer defined by $$[n_1,\cdots,n_i]=\frac{\beta_i}{\alpha_i}, \alpha_0=1.$$
Then $$\deg(u_i)=\alpha_{i-1}.$$ 
\end{proposition}
\begin{proof}
Using the induction on $i$ and the fact $\alpha_i=n_i\alpha_{i-1}+\alpha_{i-2},$ it is not difficult to show 
$$\deg(uf_{2j,f}(u))=\alpha_{2j},$$
$$\deg(uf_{2j+1,f}(u))=\alpha_{2j+1}-\alpha_{2j},$$
$$\deg(u\tilde{f}_{2j+1,f}(u))=\alpha_{2j+1},$$ and
$$\deg(u\tilde{f}_{2j,f}(u))=\alpha_{2j}-\alpha_{2j-1}.$$
Now by Corollary \ref{KL-coro} we have    $$u_{2j+1}=\pm \langle b,a_{2j,f}\rangle=\pm uf_{2j,f}$$ and $$u_{2j+2}=\pm \langle b,b_{2j+1,f}\rangle=\pm u\tilde{f}_{2j+1, f},$$ 
which completes the proof.
\end{proof}
As we have seen in the proof of Proposition \ref{prop-uPolynomial}, $P_K(u)$ equals $u_{k+1}$ of either  $C[n_1,\cdots,n_k,1]$ or   $C[n_1,\cdots,n_k,2]$. Therefore we get the following corollary.
\begin{corollary}\label{degP(u)}
Let $K=C[n_1,\cdots,n_k]$ and $[n_1,\cdots,n_k]=\frac{\beta}{\alpha}$. Then 
\begin{enumerate}
\item [\rm (i)]  $u\,|\,P_K(u)$.
\item [\rm (ii)] $\deg P_K(u)=\alpha$.
\end{enumerate}
\end{corollary}
If $K=C[n_1,\cdots,n_k]$ and  $C[n_1,\cdots,n_m], m>k$ have the same orientations on the arc corresponding to $a$ when we let the orientation of the arc corresponding to $b$ coincide,  
then the rep-polynomial $P_K(u)$ of $K$ and $u_{k+1}(u)$ of $C[n_1,\cdots,n_m]$ must be the same up to sign. If they have the opposite orientations on the arc corresponding to $a$,  $P_K(u)=\pm i^{\alpha}u_{k+1}(iu)$ when $[n_1,\cdots,n_k]=\frac{\beta}{\alpha}$. So we have the following.
\begin{corollary}\label{u_i-polynomial}
Let $K=C[n_1,\cdots,n_k]$  and $[n_1,\cdots,n_j]=\frac{\beta_j}{\alpha_j}, (\alpha_j,\beta_j)=1$. Then $u_{j+1}$ is either $P_\frac{\beta_j}{\alpha_j}(u)$ or $\pm i^{\alpha_j}P_\frac{\beta_j}{\alpha_j}(iu)$.
\end{corollary}

\begin{definition}
The $u_i$-$sequence$ of $K=C[n_1,\cdots,n_k]$ is defined as the sequence of polynomials in $\mathbb Z[u]$, $$(u_1,u_2,\cdots,u_k),$$ 
where $u_1(u)=u$.
We will call the sequence of numbers, $$(r,u_2(r),\cdots,u_k(r))$$ for a non-zero root of the rep-polynomial of $K$, the $u_i(r)$-$sequence$ of $K$.
\end{definition}
We can observe that the $u_i$-sequence of $K=C[n_1,\cdots,n_k]$ is 
\begin{equation}\label{ui-sequence}
(u,\pm\epsilon_1^{\alpha_1} P_{\frac{\beta_1}{\alpha_1}}(\epsilon_1u), \pm\epsilon_2^{\alpha_2} P_{\frac{\beta_2}{\alpha_2}}(\epsilon_2u),\cdots,\pm\epsilon_{k-1}^{\alpha_{k-1}} P_{\frac{\beta_{k-1}}{\alpha_{k-1}}}(\epsilon_{k-1}u)),
\end{equation}
where $\epsilon_j\in\{1,i\}$
for each $j=1,2,\cdots,k-1$.
\begin{remark}
For each root $r$ of $P_K(u)$ and for each $i$, 
$u_i(r)$ is related to the trace of the element $A_i$ in $SL(2,\bc)$ 
corresponding to the loop rotating the $i$-th block horizontally by 1 full turn. More precisely, the trace of $A_i$ is equal to $2-u_i(r)^2$,
$$tr A_i=2-u_i(r)^2.$$
\end{remark}
The following is an immediate consequence of the definitions for $u_i$ and $P_K(u)$.
\begin{proposition}
Let  $K$ be $C[n_1,n_2,\cdots, n_k]$ and $K'$ be its upside-down. Suppose  $P_K(u)$ and $P_{K'}(u)$ are their rep-polynomials and $(u_1,u_2,\cdots,u_k)$ and $(u'_1,u'_2,\cdots,u'_k)$ are their $u_i$-sequences. 
Then $P_{K'}(u)=P'_K(u)$ and satisfies the following.
\begin{enumerate}
\item [\rm (i)] If $P_K(r)=0$ then $P'_K(u_k(r))=0$ and $u'_k(u_k(r))^2=r^2$.
\item [\rm (ii)] If $P'_K(s)=0$ then $P_K(u'_k(s))=0$ and $u_k(u'_k(s))^2=s^2$.
\end{enumerate}
\end{proposition}
\begin{lemma}\label{reversed knot}

Let  $K$ be $C[n_1,n_2,\cdots, n_m]$ with a fixed orientation and $-K$ be the orientation-reversed diagram of $K$. Let  $P_i(u)$ and $\tilde{P}_i(u)$ be the $u_i$ of $K$ and $-K$. Then for each $i=1,\cdots,m$, $$\tilde{P}_i(u)=-P_i(-u)=\pm P_i(u).$$
\end{lemma}
\begin{proof}
If $\{a_{i,j}, b_{i,j}\}$ is a $\mathfrak{C}$-coloring on $K$, then $\{ia_{i,j}, ib_{i,j}\}$ is a $\mathfrak{C}$-coloring on $\tilde{K}$. Hence
$$\tilde{P}_i(-u)= \langle ia_{i,j}, ib_{i,j}\rangle=-\langle a_{i,j}, b_{i,j}\rangle=-P_i(u).   $$
Since $P_i(u)$ is an even polynomial or an odd polynomial by Corollary \ref{u_i-poly}, we get
$$\tilde{P}_i(u)=-P_i(-u)=\pm P_i(u).$$
\end{proof}

\begin{proposition}\label{rotation}
Let $K=C[n_1,n_2,\cdots, n_m]$ with a fixed orientation and $K'$ be the upside-down of $K$. Suppose that 
$r$ is a non-zero root of the rep-polynomial of $K$.
Then  $(u_m(r),\cdots, u_2(r),r)$ is the $u_i(u_m(r))$-sequence of $K'$ up to sign.
\end{proposition}
\begin{proof}
Suppose that $\{a_{i,j}, b_{i,j}\mid i=1,\cdots,m, \ j=0,\cdots, n_i\}$ is a $\mathfrak{C}$-coloring on $K$  such that
$\langle a_{1,0},b_{1,0}\rangle=r$, and $\{A_{i,j}, B_{i,j}\}$ are the elements in $SL(2,\mathbb C)$ which correspond to the representation. Then the outermost coloring vectors are unchanged by the half rotation up to sign, that is, 
$[a_{1,0}], [b_{1,0}], [a_{2i,j}], [a_{m,f}], [b_{m,f}]$ and $[a_{2i,f}]$ are all unchanged  for any $2i\leq m$ and $j=0,\cdots, n_{2i}$. This implies that each $u_{2i+1}=\langle a_{2i,f}, b_{1,0}\rangle$ is also unchanged by the half rotation up to sign. 

For any $i$ such that $2i\leq m$,
$$A_{2i,0}B_{2i,0}=\cdots=A_{2i,j}B_{2i,j}=\cdots=A_{2i,n_{2i}}B_{2i,n_{2i}}$$
and $A_{2i,j}B_{2i,j}$ is not changed by the half rotation. Hence the half rotations preserve  each  
$u_{2i}^2=2-tr(A_{2i,j}B_{2i,j})$.  


\end{proof}


\subsection{Non-abelian representations} \label{na-rep} 
We have seen in the previous section that the degree of the rep-polynomial  $P_K(u)$ of  $K=C[n_1,\cdots,n_k]=S(\alpha,\beta)$  is $\alpha$ and $u$ is a factor of  $P_K(u)$. Therefore if $K$ is a knot,  then $u^m\nmid P_K(u)$ for any $m>1$ and  all the roots of $\frac{1}{u}P_K(u)$ give non-abelian representations of $K$, since the degree of the  Riley Polynomial $\mathcal R(y)$ of $K$ is $\frac{\alpha-1}{2}$ and $y=u^2$. We give a direct proof for this here:

\begin{lemma}\label{ufactor}
Let $P_K(u)$ be the rep-polynomial of $K=C[n_1,\cdots,n_k]$. Then
 $u^2 \,|\, P_K(u)$ if and only if $K$ is a link. Furthermore, if $K$ is a knot then  $\frac{1}{u}P_K(u)\in \mathbb Z[u]$ is a monic polynomial whose constant term is either $1$ or $-1$.
\end{lemma}
\begin{proof} 
 Note that 
$$P_K(u)=\pm uf_{k,f}(u)  \quad \text{when} \quad k \quad  \text{is even}$$
and
$$P_K(u)=\pm u\tilde{f}_{k,f}(u)  \quad \text{when} \quad k \quad  \text{is odd}.$$
Now  we assume that  $\langle a,b\rangle=0$. Then $\{a_{i,j}, b_{i,j}\}\subset\{a, -a, b, -b\}$ for all $i,j$.
If  $K$ is a link, then 
$$a_{k,f}=\pm b\ (k: \text{even}) \quad \text{or} \quad b_{k,f}=\pm b\ (k :  \text{odd}),$$
which implies 
$$f_{k,f}(0)= 0\quad   \text{and} \quad g_{k,f}(0)= \pm 1  \quad (k: \text{even})
$$
or $$\tilde{f}_{k,f}(0)= 0\  \quad \text{and} \quad \tilde{g}_{k,f}(0)= \pm 1\quad  (k :  \text{odd}).$$ 
(See Figure \ref{u=0}.)
This proves that if $K$ is a link then $\frac{1}{u}P_K(u)|_{u=0}=0$ 
and thus $u^2$ divides $P_K(u)$. 

\begin{figure}[hbt]
\begin{center}
\scalebox{0.5}{\includegraphics{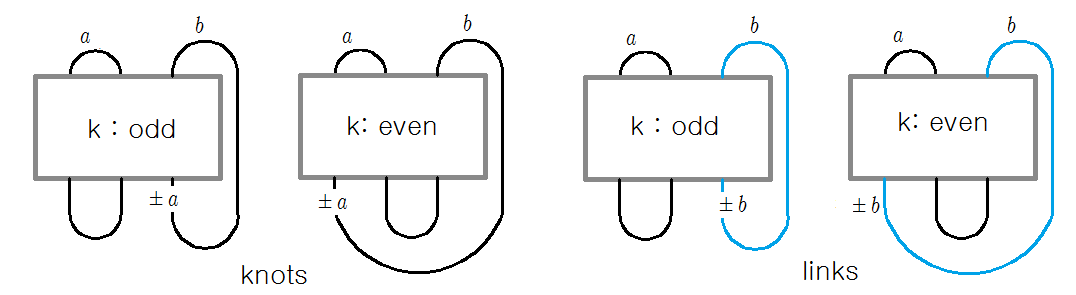}}
\end{center}
\caption{$u=0$}\label{u=0}
\end{figure}

If $K$ is a knot, 
 then  
$$a_{k,f}=\pm a\ (k: \text{even}) \quad \text{or} \quad b_{k,f}=\pm a\ (k :  \text{odd}),$$
which implies that 
\begin{equation}\label{unitsEq}
f_{k,f}(0)=\pm 1 \neq 0\  (k: \text{even}) \quad \text{or} \quad\tilde{f}_{k,f}(0)=\pm 1 \neq 0\  (k :  \text{odd})
\end{equation}
 and thus   there is $c_{2k} \in \mathbb Z, k=1,2,\cdots,\frac{\alpha-3}{2}$ such that 
\begin{equation}\label{monic-riley}
P_K(u)=u(u^{\alpha-1}+c_{\alpha-3}u^{\alpha-3}+\cdots+c_2u^2\pm1).
\end{equation}
This proves 
$$\frac{1}{u}P_K(u)|_{u=0}\neq 0,$$
and the last statement.
\end{proof}
\begin{theorem}\label{NA-rep} Let $P_K(u)$ be the rep-polynomial of a knot $K=C[n_1,\cdots,n_k]$. Then there is a 1-1 correspondence between  the set of  the squares of the roots of $\frac{1}{u}P_K(u)$ and the set of  the non-abelian representations of $K$. 
\end{theorem}
\begin{proof}
By Lemma \ref{ufactor}, all the roots of $\frac{1}{u}P_K(u)$ are nonzero. Since each root of  $P_K(u)$ corresponds to a $\mathfrak{C}$-coloring on $K$ by Proposition \ref{prop-uPolynomial},  all the roots of $\frac{1}{u}P_K(u)$ give non-abelian representations of $K$. This correspondence is bijective, because the degree of  $\frac{1}{u}P_K(u)$ is equal to that of $\mathcal{R}(u^2)$, and $\mathcal{R}(y)$ has no repeated roots.
\end{proof}

\begin{remark}\label{Riley}
If we take any Conway's normal form $K=C[n_1,\cdots,n_k]$ of a 2-bridge knot $K=S(\alpha,\beta)$ such that 
$$[n_1,n_2,\cdots,n_k]=\pm\frac{\beta}{\alpha} \,\, \text{or}\,\, \pm\frac{\alpha-\beta}{\alpha},$$ then we can obtain the Riley polynomial of $K$ from our polynomial $\frac{1}{u}P_K(u)$ by converting $u^2$ into $y$. 
That is,
\begin{equation}\label{knot-Riley}
\frac{1}{u}P_{\frac{\beta}{\alpha}}(u)=\frac{1}{u}P_{\frac{\alpha-\beta}{\alpha}}(u)=\pm\mathcal{R}(u^2)=\pm W_{11}(u^2).
\end{equation}
Note that 
the relation $W\rho(a)W^{-1}=\rho(w)\rho(a)\rho(w)^{-1}=\rho(b)$ with (\ref{knot-group}) and (\ref{parabolic-rep}) implies that  $$\left[\begin{array}{c}
W_{11}(u^2)\\
W_{21}(u^2)  
 \end{array}\right]=T(W\rho(a)W^{-1})=T(\rho(b))=\left[\begin{array}{c}
0\\
u  
 \end{array}\right],$$
 and $$\langle \begin{pmatrix}
W_{11}(u^2)\\
W_{21}(u^2)  
\end{pmatrix},\begin{pmatrix}
0\\
u  
 \end{pmatrix} \rangle=uW_{11}(u^2)=\pm P_{\frac{\beta}{\alpha}}(u).$$

For example, 
$K_1=S(7,3)$ and $K_2=S(7,5)$ are equivalent 2-bridge knots and their corresponding Conway's normal forms are $K_1=C[2,3]$ and $K_2=C[1,2,2]$, which is the upside -down of $C[2,3]$.
The rep-polynomials of $C[2,3]$ and $C[1,2,2]$ are  $u(u^6-u^4+2u^2-1)$ and $u(u^6+3u^4+2u^2-1)$, respectively. The Riley polynomials of $K_1$ and $K_2$ are $-(y^3-y^2+2y-1)$ and
 $-(y^3+3y^2+2y-1)$, respectively.

All the rep-polynomials of 2-bridge knots are expressed as combinations of Chebyshev polynomials $p_j$'s. (See Appendix.) So we can also get an explicit formula for Riley polynomial.  
\end{remark}
\begin{remark}\label{Riley-link}
For  a link $K$, we have 
\begin{equation}\label{link-Riley}
P_K(u)=\pm u^2 \mathcal{R}(u^2)=\pm u^2 W_{12}(u^2)=\pm W_{21}^*(u^2)
\end{equation}
since the degree of both $P_K(u)$ and $\pm u^2 \mathcal{R}(u^2)$ is $\alpha$. Hence there is a 1-1 correspondence between  the set of the squares of the non-zero roots of $\frac{1}{u^2}P_K(u)$ and the set of  the non-abelian representations of $K$. But we will see that $\frac{1}{u^2}P_K(u)$ might have $0$ as a root later.  

Note that  the relation $W^*\rho(a)(W^*)^{-1}=\rho(w^*)\rho(a)\rho(w^*)^{-1}=\rho(a)$
 with (\ref{parabolic-rep}) and (\ref{link-group}) implies $$\left[\begin{array}{c}
W^*_{11}(u^2)\\
W^*_{21}(u^2) 
 \end{array}\right]=T(W^*\rho(a)(W^*)^{-1})=T(\rho(a))=\left[\begin{array}{c}
1\\
0  
 \end{array}\right],$$
 and
 the relation $W\rho(b)W^{-1}=\rho(w)\rho(b)\rho(w)^{-1}=\rho(b)$
 with (\ref{parabolic-rep}) and (\ref{link-group})
 implies that
 \begin{equation*}
 \begin{split}
 ((W\rho(b)W^{-1}))^{-1})&=
 T(W \begin{pmatrix}
0 & -\frac{1}{u}\\
u & 0
 \end{pmatrix}
\rho(a) \begin{pmatrix}
0 & -\frac{1}{u}\\
u & 0
 \end{pmatrix}^{-1}W^{-1})\\
 &=
 T(W \begin{pmatrix}
0 & -\frac{1}{u}\\
u & 0
 \end{pmatrix}
\rho(a) (W\begin{pmatrix}
0 & -\frac{1}{u}\\
u & 0
 \end{pmatrix})^{-1})\\
 &=
 \left[\begin{array}{c}
uW_{12}(u^2)\\
uW_{22}(u^2)  
 \end{array}\right]=\left[\begin{array}{c}
0\\
u  
 \end{array}\right]=T(\rho(b)).
 \end{split}
 \end{equation*}
And we can observe 
$$\langle \begin{pmatrix}
W^*_{11}(u^2)\\
W^*_{21}(u^2)   
\end{pmatrix},\begin{pmatrix}
1\\
0  
 \end{pmatrix}\rangle=-W^*_{21}(u^2)=u^2W_{12}(u^2)=\langle \begin{pmatrix}
uW_{12}(u^2)\\
uW_{22}(u^2)  
 \end{pmatrix},\begin{pmatrix}
0\\
u  
 \end{pmatrix}\rangle.$$

\end{remark}
\begin{remark}
The last statement of Lemma \ref{ufactor} implies that the Riley polynomial $(-1)^{\frac{\alpha-1}{2}}\mathcal{R}(y)\in \mathbb Z[y]$ is 
a monic polynomial whose constant term is either $1$ or $-1$. 
(Note that (\ref{monic-riley}) is  equivalent to Equation 3.10 of \cite{Riley1}.)
But this does not hold for 2-bridge links. For example, the rep-polynomial of $C[2,1,2]$, the Whitehead link, is $u^4(u^4\pm 2u^2+2)$.
\end{remark}
\begin{example}
 We have seen in Example \ref{C[3]-knot} that  the rep-polynomial of $C[3]$ is $$\langle b, b_{1,3} \rangle=u(u^2-1).$$
Therefore the trefoil has only one non-abelian  parabolic representation because $1$ and $-1$ correspond to the same $\mathfrak{C}$-coloring  on $K$.  This representation has a generating meridian pair  $\begin{pmatrix}
1& 1\\
0 & 1 
 \end{pmatrix}$,  
$\begin{pmatrix}
1& 0\\
-1 & 1 
 \end{pmatrix}
$ up to conjugation.
\end{example}

\begin{example}\label{C[2,2,5]-knot}
 The rep-polynomial of a knot $K=C[2,2,5]$ is 
$$P_K(u)=u(u-1)(u+1)h(u)h(-u)$$
where $$h(u)=u^{12}-2u^{10}+u^9+4u^8-u^7-3u^6+3u^5+3u^4-u^3-u^2+2u+1$$
and all the non-zero roots of $P_K(u)$ give non-abelian representations. 

We can easily check that the $u_i$-sequence of $K=C[2,2,5]$ is
$$(u,-u^2,-u(u^4-u^2+1))=(u,-P_{C[2]}(u),-iP_{C[2,2]}(iu)).$$
 \end{example}
\begin{corollary}\label{u_i-poly}
For any $K=C[n_1,\cdots,n_k]$, $u_i\in\mathbb Z[u]$ is either an odd polynomial or an even polynomial depending on whether $K$ is a knot or a link respectively.
\end{corollary}
\begin{proof}
It follows from  (\ref{knot-Riley}), (\ref{link-Riley}), and Corollary \ref{u_i-polynomial} that
if $C[n_1,\cdots,n_{i-1}]$ is a knot, then $u_i$ is an odd polynomial in $u$ and if $C[n_1,\cdots,n_{i-1}]$ is a link, then $u_i$ is an even polynomial.
\end{proof}

\section{Trace field}
The trace field of a  representation $\rho$ is defined by
$\mathbb Q\langle tr\rho(\gamma)\,|\, \gamma\in G(K)\rangle $ \cite{MR}.
So it is obviously equal to  $\mathbb Q(y)=\mathbb Q(u^2)$ by (\ref{parabolic-rep}).  In this section, we show that 
any rep-polynomial of a 2-bridge knot always has a special decomposition, and  $u\in\mathbb  Q(u^2)$ as a unit. 
\begin{theorem}\label{u-poly-decomp}
Let $K$ be a 2-bridge knot. Then the rep-polynomial of $K$ is
$$P_K(u)=ug(u)\hat{g}(u)$$ for some $g(u) \in \mathbb Z[u]$ such that $\hat{g}(u)=(-1)^{\deg g} g(-u)$ and $\hat{g}(u)\neq g(u)$. Furtheremore there is a 1-1 correspondence between the set of roots of $g(u)$ and the set of non-abelian parabolic representations of $K$. 
\end{theorem}

\begin{proof}
Every 2-bridge knot $K$ can be expressed as  $$K=C[2n_1,2n_2,\cdots,2n_{2m}],  n_i \in \mathbb Z - \{0\},$$ so called a reduced even expansion of $K$. (See \cite{Cromwell} or \cite{GHS}.)
In such a diagram for $K$, 
\begin{enumerate}
\item [\rm (i)] $u_{2k}=\pm P_{C[2n_1,2n_2,\cdots,2n_{2k-1}]}(u)$ is an even polynomial in $u$,
\item [\rm (ii)] $u_{2k+1}=\pm P_{C[2n_1,2n_2,\cdots,2n_{2k}]}(u)$ is an odd polynomial in $u$,
\end{enumerate}
as one can easily see that $C[2n_1,2n_2,\cdots,2n_{2k}]$ is a knot and $C[2n_1,2b_2,\cdots,2n_{2k-1}]$ is a link for any $k$ such that $2k+1<2m$.
For example, $u_2=-u^2p_{n_1}(-2-u^2)$. (See Appendix \ref{EvenExpansion} for the details.)

So $u_{2m}=f(u^2)$ for some $f \in \mathbb Z[u]$. Now if we let $\pm r_1, \cdots, \pm r_k$ be the non-zero roots of the rep-polynomial $P_K(u)$, then the
non-zero 
roots of the rep-polynomial $P'_K(u)$ which comes from the upside-down diagram, or equivalently the rep-polynomial coming from the arc coloring which
 starts with the two vectors on the bottom, are $$\pm f(r_1^2), \cdots,\pm f(r_k^2).$$  By  the fact that all the roots of the Riley polynomial are distinct \cite{Riley1}, $r_i\neq r_j$ for $i\neq j$, and thus 
$$P_K(u)= u(u+r_1)(u-r_1) \cdots(u+r_k)(u-r_k).$$  
 and   $$ P'_K(u)= u(u+f(r_1^2))(u-f(r_1^2)) \cdots(u+f(r_k^2))(u-f(r_k^2)).$$

Since $$P_K(u)= u(u+r_1)(u-r_1) \cdots(u+r_k)(u-r_k)= u(u^2-r_1^2)\cdots(u^2-r_k^2)\in \mathbb Z[u],$$
$$(u-r_1^2)\cdots(u-r_k^2)\in \mathbb Z[u].$$  
As elementary symmetric polynomials of $f(a_1), f(a_2), \cdots,f(a_k)$ can be expressed as those of $ a_1,a_2,\cdots,a_k$, 
$$g'(u):=(u-f(r_1^2)) \cdots(u-f(r_k^2))\in \mathbb Z[u],$$  
$$ P'_K(u)=ug'(u)\hat{g'}(u).$$

Similarly, since $P_K=(P')'_K$ there is a polynomial $h$ with integer coefficients such that $$ P_K(u)= u(u+h(f(r_1^2)))(u-h(f(r_1^2))) \cdots(u+h(f(r_k^2)))(u-h(f(r_k^2)))$$  and $g(u):=(u-h(f(r_1^2))) \cdots (u-h(f(r_k^2))) \in \mathbb Z[u]$,  and $$P_K(u)=ug(u)\hat{g}(u)$$ as desired. The property $\hat{g}(u)\neq g(u)$ is obvious, because if it is not the case  then $P_K(u)=ug(u)^2$ and this contradicts that all the roots of the Riley polynomial are distinct. 

Since two $\mathfrak{C}$-colorings  on $K$ with $u=\langle a_{1,0},b_{1,0} \rangle$ and $-u=\langle a_{1,0},b_{1,0} \rangle$ correspond to the same representation of $K$, the last statement follows from Theorem \ref{NA-rep}.
\end{proof}

\begin{example}\label{2-3}
 We have seen in Remark \ref{Riley} that  the rep-polynomial of $C[2,3]$ and $C[1,2,2]$ are  $u(u^6-u^4+2u^2-1)$ and $u(u^6+3u^4+2u^2-1)$, respectively. We can check that
$$u^6-u^4+2u^2-1=(u^3+u^2-1)(u^3-u^2+1)$$
and $$u^6+3u^4+2u^2-1=(u^3+u^2+2u+1)(u^3-u^2+2u-1).$$
\end{example}
\begin{remark}
There exist  $2^{\frac{\alpha-1}{2}}$ number of polynomials $g(u)\in \mathbb C[u]$ such that 
$P_K(u)=ug(u)\hat{g}(u)$ for a 2-bridge knot $K=S(\alpha,\beta)$. But Theorem \ref{u-poly-decomp} implies that we can find one among them which has integer coefficients.    
\end{remark}
\begin{remark}
Theorem \ref{u-poly-decomp} does not hold for the link case.  For example, the rep-polynomial of  $C[2,1,2]$, the Whitehead link,  is
$$P_K(u)=u^4(u^4\pm 2u^2+2)=u^4(u^2+1\pm i)(u^2-1\pm i).$$ 
The  sign in the above equation depends on the orientations of the two components of $K$. If we change the orientation of one of two  components, then $u$ is changed to $\pm iu$.
As this example shows, Riley polynomial $W_{12}(u^2)=\frac{1}{u^2}P_K(u)$ might have $0$ as a root in link case.

Some rep-polynomials of links are decomposed into two polynomials with integer coefficients, but the two factors do not usually have the same property  as $g(u)$ in Theorem \ref {u-poly-decomp}.
The rep-polynomials of a link  
$K=C[2,1,4]$ are $P_K(u)$ or $ P_K(iu)$, where 
$$P_K(u)=u^2(u^6-3u^4+4u^2-1)(u^6-u^4+1).$$
\end{remark}
When $K$ is a knot,  $\frac{1}{u}P_K(u)\in \mathbb Z[u]$ is a monic polynomial whose constant term is either $1$ or $-1$ by Lemma \ref{ufactor}. So it is easy to show that $r^2$ is a unit in $\mathbb  Q(r^2)$ for each nonzero root $r$ of $P_K(u)$,
and  we can show using Theorem \ref{u-poly-decomp} that  $r\in\mathbb  Q(r^2)$ as follows.

\begin{proposition}\label{units} For any 2-bridge knot $K$, each nonzero root $r$  of $P_K(u)$ belongs to the trace field $\mathbb Q(r^2)$, as a unit, of the  parabolic representation corresponding to $r$.
\end{proposition}
\begin{proof}
Let $K=C[n_1,\cdots,n_k]$.
By Lemma \ref{ufactor},  
 there is $c_{2k} \in \mathbb Z, k=1,2,\cdots,\frac{\alpha-3}{2}$ such that 
$$P_K(u)=u(u^{\alpha-1}+c_{\alpha-3}u^{\alpha-3}+\cdots+c_2u^2\pm1).$$
Therefore each nonzero root $r$ of $P_K(u)$ satisfies the following equation:
\begin{equation}\label{unit}
r^2(r^{\alpha-3}+c_{\alpha-3}r^{\alpha-5}+\cdots+c_2)=\pm1,
\end{equation}
which implies that 
$r$ is a unit in $\mathbb  Q(r^2)$ if $r\in\mathbb  Q(r^2)$. 

 By Theorem \ref{u-poly-decomp}, $P_K(u)=\pm ug(u)g(-u)$ and thus there are $A(u)$ and $B(u)$ in $\mathbb Z[u]$ such that $g(u)=A(u^2)+uB(u^2)$.  Note that  $\hat{g}(u)\neq g(u)$ and thus $B(u^2)\neq 0$.
So if $g(r)=0$ then $r=-\frac{A(r^2)}{B(r^2)} \in \mathbb  Q(r^2)$, which implies that $\mathbb  Q(r)=\mathbb  Q(r^2)$. 
\end{proof}
\begin{remark}
Proposition \ref{units} does not hold for 2-bridge links. For example, the rep-polynomial of $C[2,1,2]$ is $u^4(u^4\pm 2u^2+2)$, and  $\mathbb  Q(r)$ is not equal to $\mathbb  Q(r^2)$ for any $r$ such that $r^4\pm2r^2+2=0$. 
\end{remark}
\begin{corollary} \label{trace-field}
Let $C[n_1,n_2,\cdots, n_k]$ and $r$ be a non-zero root of the rep-polynomial $P_K(u)$ of $K$. Then
 $\mathbb{Q}(r)=\mathbb{Q}(u_k(r))$.
\end{corollary}

\section{Complex volume and Cusp shape}
In this section, we will see that  the complex volume and the cusp shape  of a parabolic representation of  an arbitrary 2-bridge knot can be easily computed using the quandle coloring. 

Let $C[n_1,\cdots,n_k]$ be any Conway diagram of a 2-bridge knot $K$ with an arc coloring $\{a_{i,j}, b_{i,j}\}$
which corresponds to a parabolic representation $\rho : G(K) \rightarrow SL(2,\bc)$ and $n=n_1+\cdots+n_k$. Then we have $n+2$ regions, $r_1,\cdots,r_{n+2}$ and we can define a \emph{region coloring} 
$\beta : \{r_1,\cdots,r_{n+2} \} \rightarrow \mathbb C^2$ on the given Conway diagram, satisfying the condition illustrated in Figure~\ref{fig:region} around each arc \cite{IK}. 
\begin{figure}[hbt]
	\begin{center}
		\scalebox{0.4}{\includegraphics{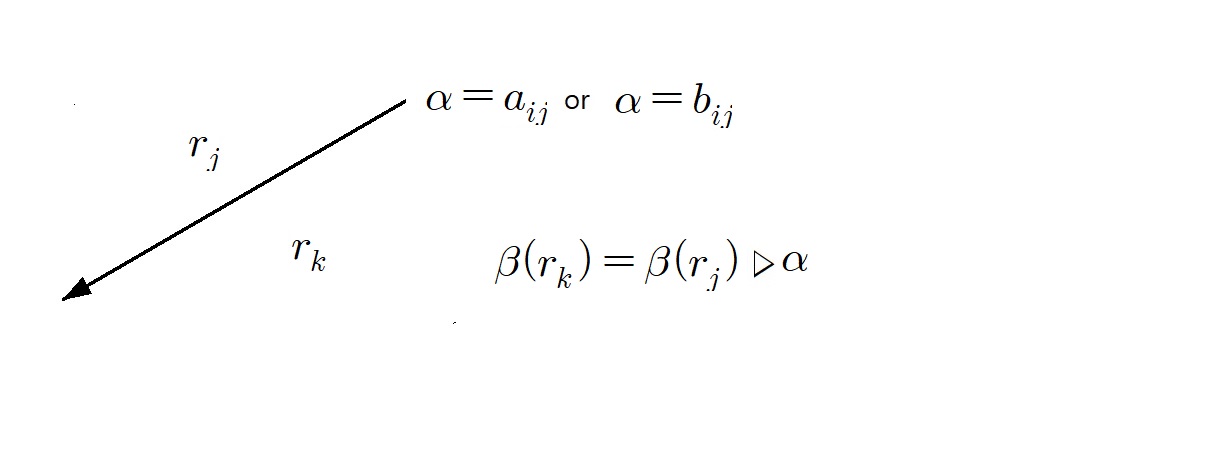}}
	\end{center}
	\caption{Rule for a region coloring}
	\label{fig:region}
\end{figure}
Now we choose any non-zero generic vector $p\in \mathbb C^2$  and assign a complex number $w_j=\langle p, \beta(r_j) \rangle$ to each region $r_j$, which is called \emph{a region variable}, 
and define a function $C(w_1,\cdots,w_{n+2})$ by the sum of
\begin{equation*}
\left\{
\begin{array}{ll}
\dfrac{w_a w_c -w_b w_d}{(w_a-w_d)(w_c-w_b)} -1 & \textrm{for Figure \ref{fig:crossings} (left)} \\[15pt]
\dfrac{w_a w_c -w_b w_d}{(w_a-w_d)(w_c-w_b)} +1 & \textrm{for Figure \ref{fig:crossings} (right)} 
\end{array}			
\right.
\end{equation*} over all crossings \cite{KKY1}.
\begin{figure}[hbt]
	\begin{center}
		\scalebox{0.4}{\includegraphics{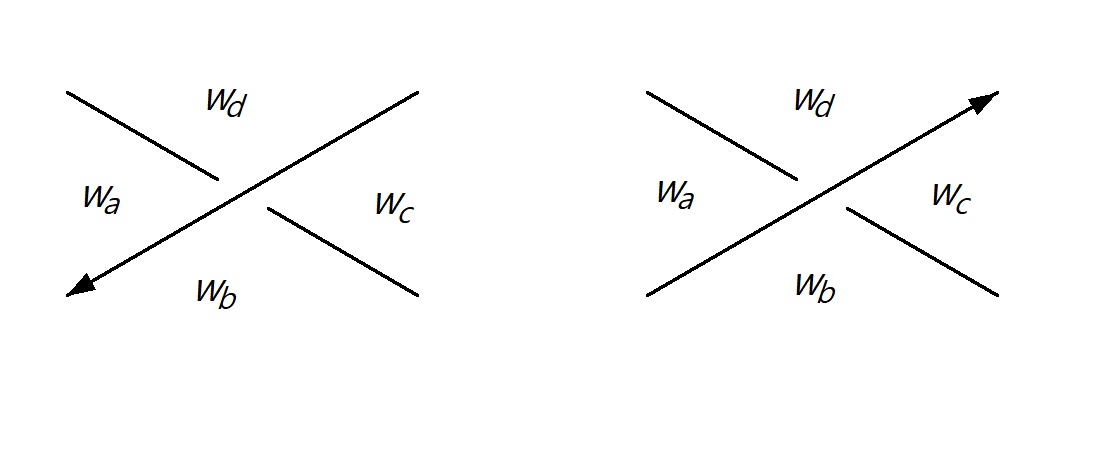}}
	\end{center}
	\caption{Region variables at a crossing}
	\label{fig:crossings}
\end{figure}	
The \emph{potential function} $W(w_1,\cdots,w_{n+2})$ which is defined by  Cho and Murakami  in \cite{Cho2, Cho-M},  is the sum of
	\begin{equation*}
\left\{
\begin{array}{ll}
-\textrm{Li}_2 \left(\dfrac{w_d}{w_a} \right) -\textrm{Li}_2 \left(\dfrac{w_d}{w_c} \right) + \textrm{Li}_2 \left(\dfrac{w_a}{w_b} \right) + \textrm{Li}_2 \left(\dfrac{w_c}{w_b}\right) &\\[10pt]
\mkern 100mu+ \textrm{Li}_2 \left(\dfrac{w_b w_d}{w_a w_c} \right)  - \dfrac{\pi^2}{6} + \textrm{log} \dfrac{w_a}{w_b} \ \textrm{log} \dfrac{w_c}{w_b}  & \mkern 10mu \textrm{for Figure \ref{fig:crossings} (left)} \\[15pt]

\textrm{Li}_2 \left(\dfrac{w_a}{w_b} \right) +\textrm{Li}_2 \left(\dfrac{w_a}{w_d} \right) - \textrm{Li}_2 \left(\dfrac{w_b}{w_c} \right) - \textrm{Li}_2 \left(\dfrac{w_d}{w_c}\right) &\\[10pt]
\mkern 100mu - \textrm{Li}_2 \left(\dfrac{w_a w_c}{w_b w_d} \right)  + \dfrac{\pi^2}{6} - \textrm{log} \dfrac{w_b}{w_c} \ \textrm{log} \dfrac{w_d}{w_c}  & \mkern 10mu\textrm{for Figure \ref{fig:crossings} (right)} \\[15pt]
\end{array}			
\right.
\end{equation*} over all crossings. 

Using the above two functions defined on the region variables induced from the arc coloring vectors, the cusp shape  and the complex volume $\textrm{Vol}_\mathbb C(\rho) \in \mathbb C/i\pi^2 \mathbb Z$
can be easily computed  as follows  for any parabolic representation $\rho$ of $K$ \cite{Cho2, Cho-M,KKY2}.
 (See also \cite{Zickert} and \cite{GTZ} for the earlier relevant works for complex volume.)
\begin{theorem}\label{cusp-volume} Let $\{a_{i,j}, b_{i,j}\}$ be an arc coloring of a $2$-bridge knot $K=C[n_1,\cdots,n_k]$ corresponding to a  parabolic representation $\rho: G(K) \rightarrow SL(2,\bc)$. Let $\{w_i\}_{i=1,\cdots,n+2}$ be any region variable induced from the arc coloring, where $n=\sum_{j=1}^kn_j$.
Then
\begin{enumerate}
\item [\rm (i)]  the cusp shape of $\rho$ is given by
	$C(w_1,\cdots,w_{n+2})$,
\item [\rm (ii)] the complex volueme of $\rho$ is given by
	$$i \,\textrm{Vol}_\mathbb C (\rho) \equiv W_0( w_1,\cdots,w_{n+2}) \quad (\emph{mod } \pi^2 \mathbb Z).$$ 
\end{enumerate}
Here the function $W_0$ is defined as follows.
	$$W_0(w_1,\cdots,w_{n+2}):=W(w_1,\cdots,w_{n+2}) - \sum_{k=1}^{n+2} \left( w_k \dfrac{\partial W}{\partial w_k} \right) \textrm{Log}\, w_k$$
\end{theorem}
\begin{example}
Let $K_1=C[2,3,0,3,2,-2,2,3]$ and $K_2=C[2,3]$. Then by Theorem 6.1, Proposition 6.2 and Remark 6.3 of \cite{ORS}, there is a proper branched fold map $f : (\mathbb S^3,K_1) \rightarrow (\mathbb S^3,K_2)$ which respects the bridge structures, which induces an epimorphism $f_*: G(K_1) \rightarrow G(K_2)$, and the degree of $f_*$ is equal to $3$.

 We have seen in Example \ref{2-3} that  the rep-polynomial $P_{\frac{3}{7}}(u)$ of $K_2$ is  as follows:
$$P_{\frac{3}{7}}(u)=u(u^6-u^4+2u^2-1)=u(u^3+u^2-1)(u^3-u^2+1).$$
 The rep-polynomial $P_{\frac{101}{217}}(u)$ of $K_1$ is  expressed as
 $$P_{\frac{101}{217}}(u)=u(u^3+u^2-1)(u^3-u^2+1)h(u)\hat{h}(u),$$
 where $h(u)$ is an irreducible monic polynomial of degree $105$: $$h(u)=u^{105}-2u^{104}-5u^{103}+14u^{102}+\cdots-4u^2+1.$$
 We can observe that $P_{\frac{3}{7}}(u)\,|\,P_{\frac{101}{217}}(u)$. We will see in the next section that it is generally true that if $K_1\geq K_2$ then  $P_{K_2}(u)$  divides either  $P_{K_1}(u)$ or $P'_{K_1}(u)$ (see Theorem \ref{KM}).

Now we compute the cusp shapes and the complex volumes of the  parabolic representations of $G(K_1)$ and $G(K_2)$ which correspond to the non-zero roots of $P_{\frac{3}{7}}(u)$ using Theorem \ref{cusp-volume}. By the calculation using Mathematica, we can  check  that  
 the cusp shapes and the complex volumes of $K_1$ are exactly $3$ times those of $K_2$ as expected, respectively. (See Table \ref{cusp shape} and Table \ref{cx volume}.)
 \begin{table}[hbt]
	\begin{center}
\begin{tabular}{|c|c|c|}	\hline
	non-zero roots of $P_{\frac{3}{7}}(u)$&$K_1=C[2,3,0,3,2,-2,2,3]$ & $K_2=C[2,3]$\\ \hline
	$u= 0.75487766$ &$ -27.05853199$&$ -9.01951066$\\ \hline
	$u=0.87743883 + 0.74486176i$&$ -7.47073400+8.93834119 i$&$ -2.49024466+2.97944706 i$\\ \hline
		$u= 0.87743883 - 0.74486176i$&$ -7.47073400-8.93834119 i$&$ -2.49024466-2.97944706 i $\\ \hline
	\end{tabular}
\end{center}\caption{cusp shapes of $K_1$ and $K_2$}\label{cusp shape}
\end{table}

 \begin{table}[hbt]
	\begin{center}
\begin{tabular}{|c|c|c|}	\hline
	non-zero roots of $P_{\frac{3}{7}}(u)$&$K_1=C[2,3,0,3,2,-2,2,3]$ & $K_2=C[2,3]$\\ \hline
	$u= 0.75487766$ &$0+3.34036365i$&$ 0+1.11345455i$\\ \hline
	$u=0.87743883 + 0.74486176i$&$-8.48436626+30.40603247i $&$-2.82812208-3.02412837i $\\ \hline
		$u= 0.87743883 - 0.74486176i$& $8.48436626-9.07238513i$ & $2.82812208-3.02412837i$\\ \hline
	\end{tabular}
\end{center}\caption{complex volumes of $K_1$ and $K_2$}\label{cx volume}
\end{table}
\end{example}
Note that $cs(K_1)=3\,cs(K_2)+4\pi^2$ when $u=0.8774388331 + 0.7448617666i$ and hence 
$$cs(K_1)=3\,cs(K_2)\quad (\textrm{mod } \pi^2 \mathbb Z).$$ (See Table \ref{cx volume}.)

\section{Epimorphisms between knot groups}  
There is a partial order on the set of prime knots as follows : We write $K_1 \geq K_2$ for two prime knots if there exists an epimorphism from $G(K_1)$ to $G(K_2)$. A knot is called $minimal$ if its knot group admits epimorphisms onto the knot groups of only the trivial knot and itself.

\subsection{ORS-expansion}
Ohtsuki, Riley and Sakuma have constructed in \cite{ORS} systemetically  epimorphisms between 2-bridge knot groups preserving peripheral structure when the two knots have some special continued fraction expansions, and then it can be shown that all epimorphisms between 2-bridge knot groups arise only from those  Ohtsuki-Riley-Sakuma construction 
by the result of Aimi-Lee-Sakai-Sakuma \cite{ALS}. (See \cite{Suzuki} and also \cite{Agol}
.) Therefore non-minimal 2-bridge knots have the following special Conway's normal forms. 
\begin{definition}
We say $K$ has an $ORS$-$expansion$ $of$ $type$ $n$ with respect to ${\bf a}=(a_1, a_2,\cdots,a_m)$ 
 if $K$ can be written as 
$$K=C[\epsilon_1{\bf a}, 2c_1, \epsilon_2 {\bf a^{-1}}, 2c_2, \epsilon_3 {\bf a},2c_3, \epsilon_4 {\bf a^{-1}},2c_4,\cdots, \epsilon_{n-1} {\bf a^{-1}},2c_{n-1}, \epsilon_{n} {\bf a^{(-1)^{n+1}}}]$$ 
where
$${\bf a^{-1}}=(a_m, a_{m-1},\cdots,a_1),\ \epsilon_i=\pm 1\ (\epsilon_1=1), \ c_i \in \mathbb Z.$$
\end{definition}
	\begin{remark}
	Let $K$ be an ORS-expansion of type $n$ with respect to ${\bf a}=(a_1, a_2,\cdots,a_m)$. Then 
	\begin{enumerate}
\item	$K$ is a knot if $n$ is odd and it is a link if $n$ is even. 
\item We can exclude the case $c_i=0,\ \epsilon_i \epsilon_{i+1}=-1$, and this expansion with respect to ${\bf a}$ is unique (see \cite{Suzuki} and \cite{GHS} for details).
\item If $c_i=0$ and $\epsilon_i \epsilon_{i+1}=1$, then we can reduce the length of the expansion by 2. Thus  the resulting expansion after doing all the possible reducing, is the reduced even expansion of $K$, if ${\bf a}$ is a reduced even expansion.  
\item $C[a_m, a_{m-1},\cdots,a_1]$ is equivalent to $C[a_1, a_2,\cdots,a_m]$ if $m$ is odd and it is equivalent to the mirror of $C[a_1, a_2,\cdots,a_m]$ if $m$ is even. It follows from the fact 
the upside-down of $C[a_1, a_2,\cdots,a_m]$ is $C[(-1)^{m+1}a_m, (-1)^{m+1}a_{m-1},\cdots,(-1)^{m+1}a_1].$
\end{enumerate}
\end{remark}

\begin{figure}[hbt]
\begin{center}
\scalebox{0.3}{\includegraphics{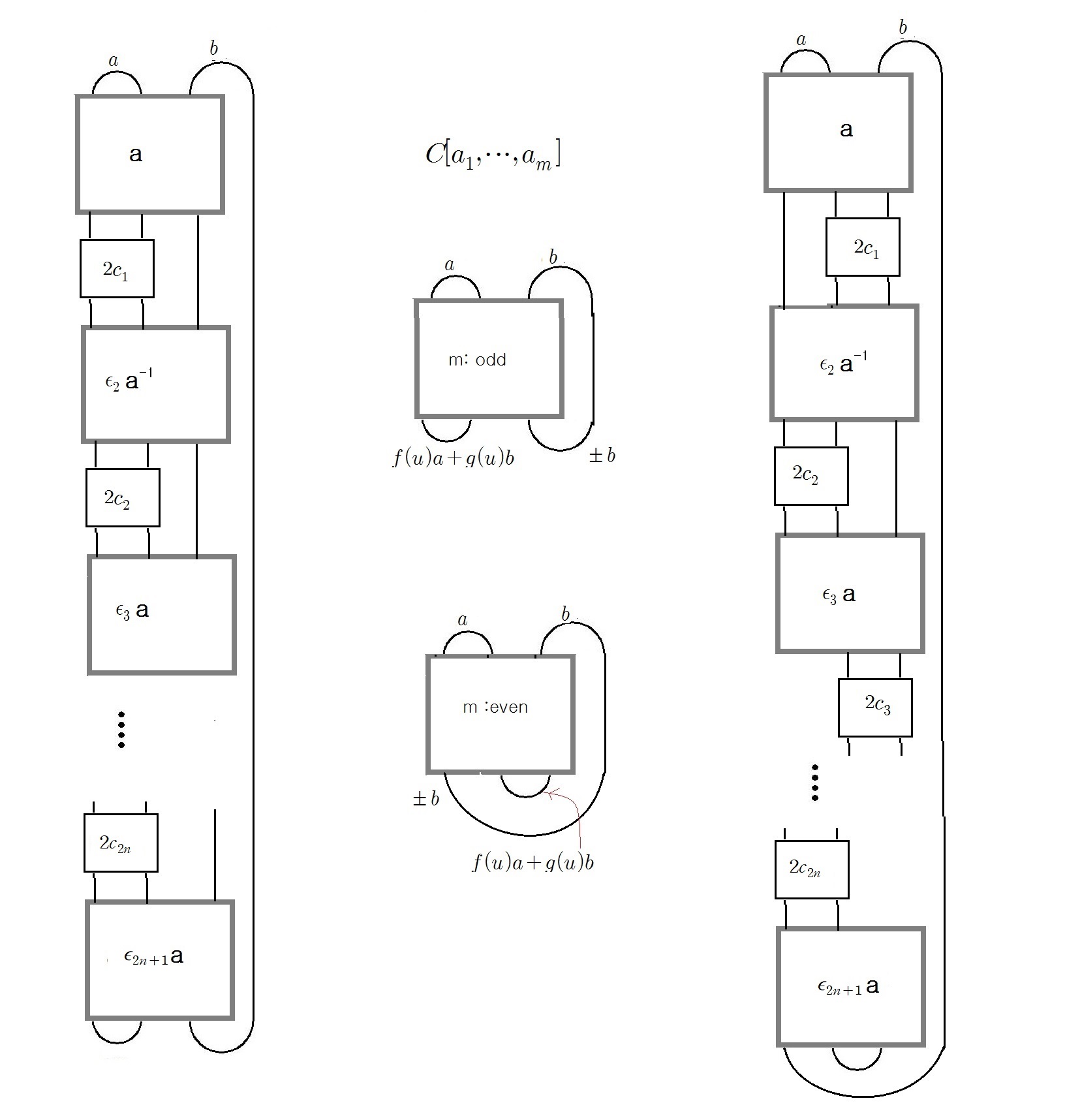}}
\end{center}
\caption{$C[\epsilon_1{\bf a}, 2c_1, \epsilon_2 {\bf a^{-1}}, 2c_2, \epsilon_3 {\bf a},2c_3,\cdots, 2c_{2n}, \epsilon_{2n+1} {\bf a}]$: knots}\label{prop-proof}
\end{figure}

\begin{theorem}\label{expansion}
Let $K$ be a 2-bridge link which has 
an ORS-expansion of $type$ $n$ with respect to ${\bf a}=(a_1, a_2,\cdots,a_m)$. 
Suppose that $(u_1,\cdots,u_m)$ is the $u_i$-sequence of $A=C[a_1, a_2,\cdots,a_m]$.
Then 
\begin{enumerate}
\item [\rm (i)] 
the rep-polynomial $P_K(u)$ of $K$ has  the rep-polynomial $P_A(u)$ of $A$ as a factor if $K$ is a knot, and either $P_K(u)$ or $P_K(iu)$ has $P_A(u)$ as a factor if $K$ is a link. 
\item [\rm (ii)] 
 the $u_i$-sequence of $K$ is as follows  when 
 it is considered in $\mathbb Z[u]/(\frac{1}{u}P_{A}(u))$.
$$(u_1,\cdots,u_m,0,\pm u_m,\cdots,\pm u_1,0,\pm u_1,\cdots,\pm u_m,0,\cdots, 0,\pm u_1,\cdots,\pm u_m)$$
where the sign is not determined.
\end{enumerate}
\end{theorem}
\begin{proof}
We may assume that the orientation of $C[a_1, a_2,\cdots,a_m]$ is the same as the first part's orientation of $K$, by changing the orientation of one component of a link $K$ if necessary, because each link has two rep-polynomials up to its orientation and one is obtained from the other by converting $u$ to $iu$. 

Consider a $\mathfrak C$-coloring on $K$ which starts with two vectors $a, b$ such that $u=\langle a,b \rangle$ is a root of $P_A(u)$. 
Then the last two vectors of  $C[a_1, a_2,\cdots,a_m]$ are $f(u)a+g(u)b, \pm b$  for some polynomials $f(u)$ and $g(u)$, and all the vectors of $2c_1$-block are $\pm (f(u)a+g(u)b)$ by Lemma \ref{KeyL}. (See Figure \ref{prop-proof}.) So the first two vectors of $\epsilon_2 {\bf a^{-1}}$-blocks are $\pm (f(u)a+g(u)b), \pm b$.
We claim that the last two vectors of $\epsilon_2 {\bf a^{-1}}$-blocks are $\pm (a+h_2(u)b), \pm b$ for some $h_2(u)\in\mathbb Z[u]$. If $\epsilon_2=-1$, then these blocks are the horizontally reflected diagram of $\epsilon_1{\bf a}$-block with an reversed orientation and thus all the coloring vectors are also reflected up to sign and especially the last two vectors are $\pm a, \pm b$. If $\epsilon_2=1$, then these blocks are the horizontally half-rotated diagram of $\epsilon_1{\bf a}$-blocks with the reversed orientation 
 and thus by  Lemma \ref{reversed knot}  and Proposition \ref{rotation}
 there is such $h_2(u)\in\mathbb Z[u]$. See Figure \ref{rotation_reversed}.
\begin{figure}[hbt]
\begin{center}
\scalebox{0.3}{\includegraphics{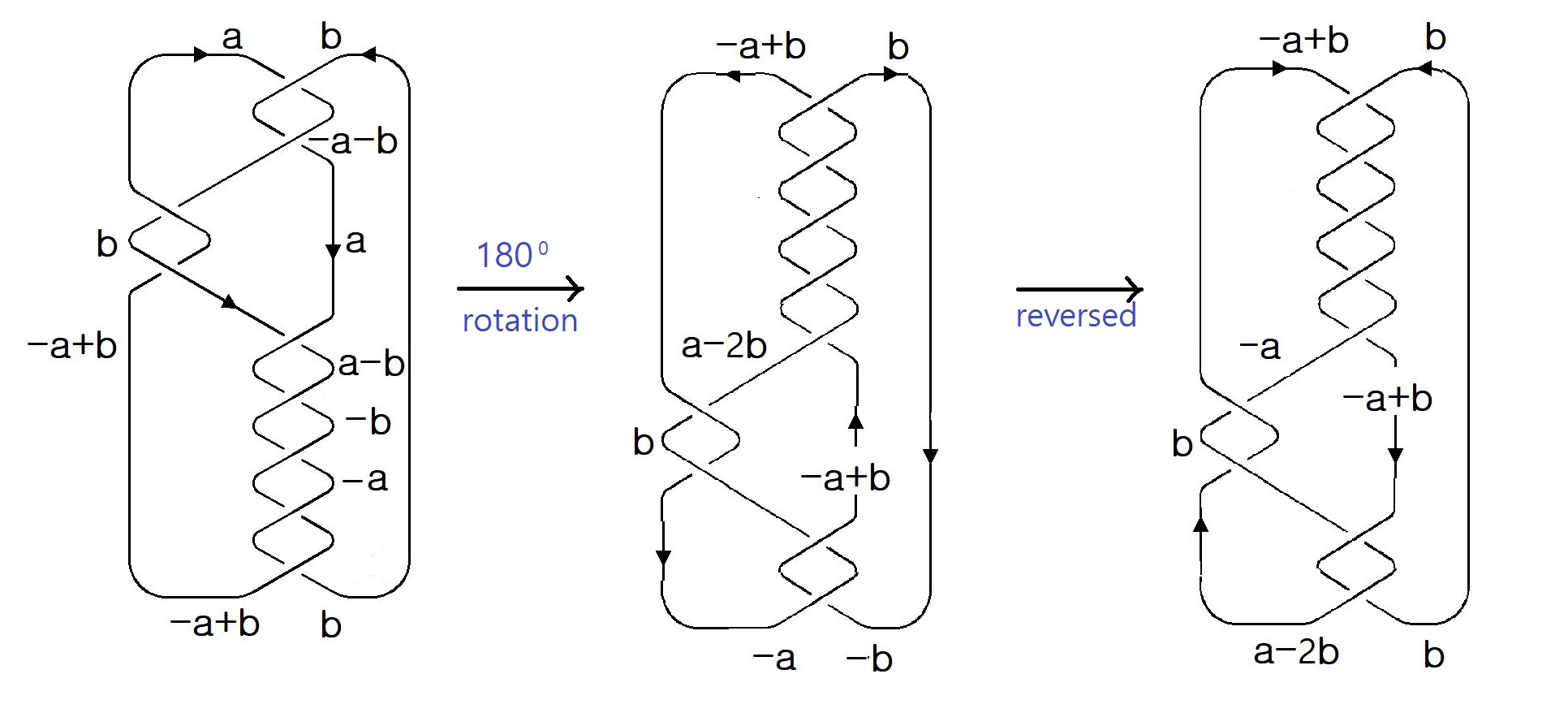}}
\end{center}
\caption{when $\epsilon_2=1$ ; $\langle a,b\rangle=1$}\label{rotation_reversed}
\end{figure}

By Lemma \ref{KeyL} again, all the vectors of $2c_2$-block are $\pm (a+h_2(u))b$ and thus the first two vectors of $\epsilon_3 {\bf a}$-block are $\pm (a+h_2(u))b, \pm b$. 
Since $\langle a+h_2(u)b, b\rangle=\langle a,b\rangle=u$, 
the last two vectors of $\epsilon_3 {\bf a}$-block  must be of the form $f_3(u)a+g_3(u)b, \pm b$ for some $f_3(u), g_3(u)\in \mathbb Z[u]$ satisfying 
$$
u_m(u)=\pm \langle f(u)a+g(u)b, b\rangle = \pm\langle f_3(u)a+g_3(u)b, b\rangle,
$$
(We see that $f_3(u)$ is necessarily equal to either $f(u)$ or $-f(u)$.)

By repeating this process, we can conclude that the last two vectors of the diagram 
$$C[\epsilon_1{\bf a}, 2c_1, \epsilon_2 {\bf a^{-1}}, 2c_2, \epsilon_3 {\bf a},2c_3, \epsilon_4 {\bf a^{-1}},2c_4,\cdots, \epsilon_{2n-1} {\bf a},2c_{2n-1}, \epsilon_{2n} {\bf a^{-1}}]$$  are $\pm (a+h_{2n}(u)b), \pm b$, and the last two vectors of the diagram 
$$C[\epsilon_1{\bf a}, 2c_1, \epsilon_2 {\bf a^{-1}}, 2c_2, \epsilon_3 {\bf a},2c_3, \epsilon_4 {\bf a^{-1}},2c_4,\cdots, \epsilon_{2n} {\bf a^{-1}},2c_{2n}, \epsilon_{2n+1} {\bf a}]$$  are $f_{2n+1}(u)a+g_{2n+1}(u)b, \pm b$ such that 
$$
\langle f(u)a+g(u)b, b\rangle = \pm\langle f_{2n+1}(u)a+g_{2n+1}(u)b, b\rangle
.$$
This implies that   $u$ is also a root of the rep-polynomial of $K$, which 
implies that $P_A(u)$ divides $P_K(u)$. This proves (i).

The statement (ii)  is obvious from Lemma \ref{KeyL}, Lemma \ref{reversed knot}, and Lemma  \ref{rotation}.
\end{proof}

\begin{example}
A link $C[3,2,3]$ and a knot $C[3,2,6]=C[3,2,3,0,3]$ are ORS-expansions of type $2$ and $3$ with respect to ${\bf a}=(3)$. Therefore the rep-polynomial of $C[3,2,6]$ is divided by the rep-polynomial of $C[3]$, $u(u^2-1)$, and so is the rep-polynomial of $C[3,2,3]$  if we choose a suitable orientation.
These are the first 2 diagrams of Figure \ref{PA-diagram}.  
\begin{figure}[hbt]
	\begin{center}
		\scalebox{0.4}{\includegraphics{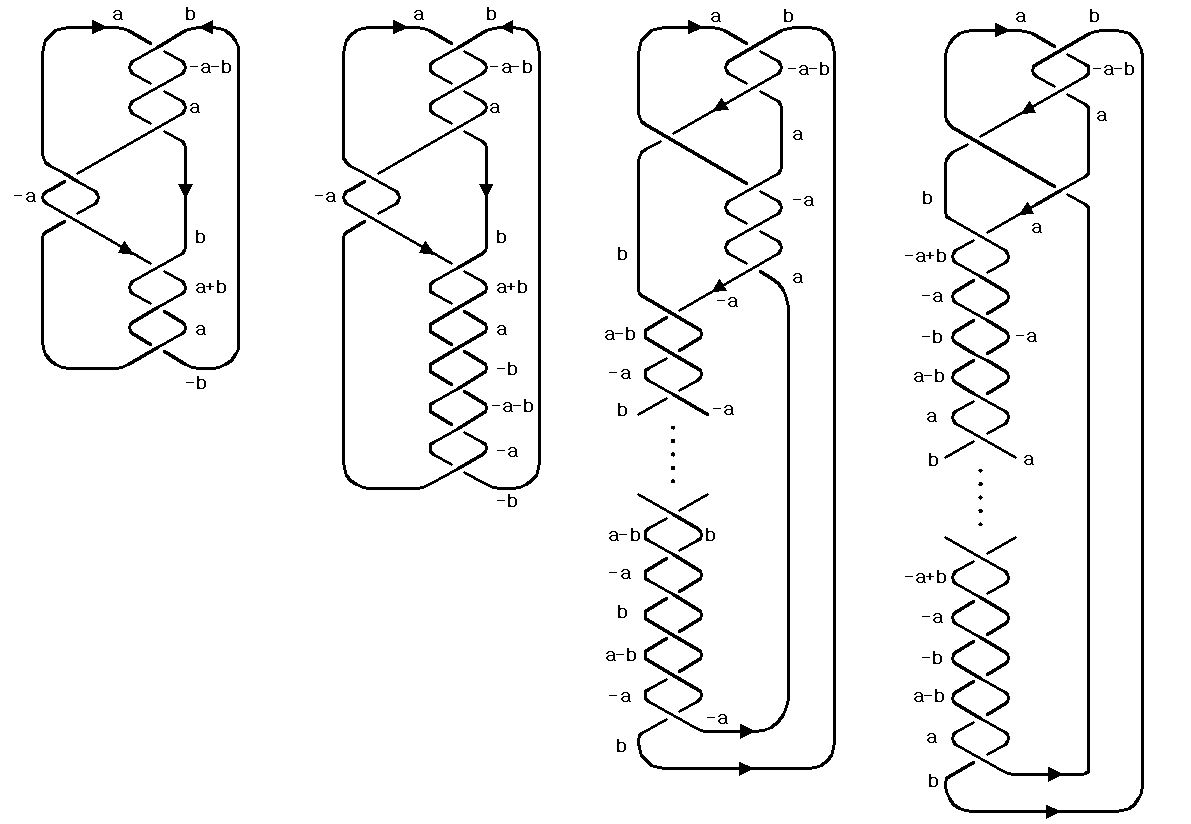}}
	\end{center}
	\caption{Knots or links $\geq 3_1$; $\langle a,b \rangle=1$}\label{PA-diagram}
\end{figure}
We can check that if the orientation of one component of a link $C[3,2,3]$ is reversed, then $iP_{C[3]}(iu)=u(u^2+1)$ divides the rep-polynomial $P_K(u)$ of $K=C[3,2,3]$, or equivalently, $P_{C[3]}(u)$ divides $P_K(iu)$.  (See Figure \ref{PA-Link-diagram}.)
Note that the two rep-polynomials of $K=C[3,2,3]$ is 
$$P_1(u)=u^2(u^2-1)^3(u^2-2)(u^6-3u^4+2u^2+2)(u^8-2u^6+2u^2+1)$$
and
$$P_2(u)=u^2(u^2+1)^3(u^2+2)(u^6+3u^4+2u^2-2)(u^8+2u^6-2u^2+1),$$
and  $u(u^2-1)$ is a factor of $P_1(u)=\pm P_2(iu)$.
\end{example}
\begin{figure}[hbt]
	\begin{center}
		\scalebox{0.5}{\includegraphics{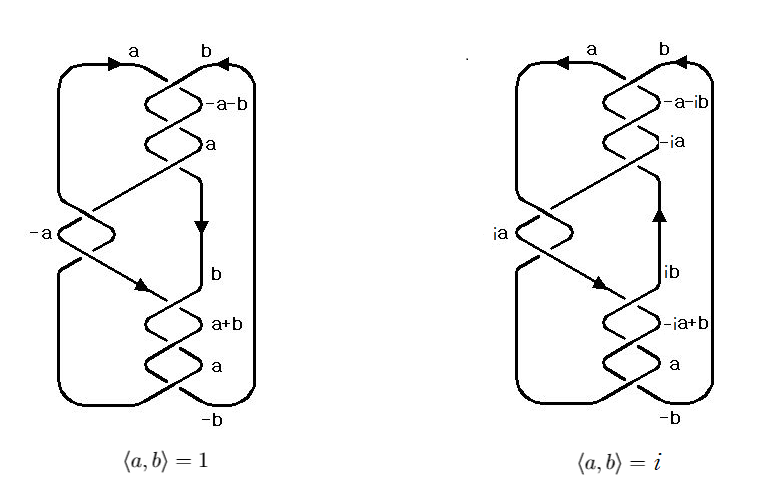}}
	\end{center}
	\caption{$C[3,2,3]$-link $\geq 3_1$}\label{PA-Link-diagram}
\end{figure}
Kitano and Morifuji proved in \cite{KM} that $K_1 \geq K_2$ if the Riley polynomial of $K_2$ divides that of $K_1$. 
Therefore the following Ohtsuki-Riley-Sakuma's result immediately follows from Theorem \ref{expansion}.
\begin{corollary}[Ohtsuki-Riley-Sakuma, \cite{ORS}]\label{ORS-theorem}
Let $K$ be a 2-bridge knot or link which has 
an ORS-expansion of $type$ $n$ with respect to ${\bf a}=(a_1, a_2,\cdots,a_m)$. Then $K\geq A=C[a_1, a_2,\cdots,a_m]$.
\end{corollary}

\subsection{Epimorphisms and Riley polynomials}
In this subsection, we prove the converse statement of the result of Kitano and Morifuji.
To do this, we need to prove the following.  
\begin{theorem}\label{KM}Let  $K_1$ and $K_2$ be 2-bridge knots. Then the followings are equivalent.
\begin{enumerate}
		\item [\rm (i)] $K_1\geq K_2$. 
        \item [\rm (ii)] $P_{K_2}(u)$  divides either  $P_{K_1}(u)$ or $P'_{K_1}(u)$.
        \end{enumerate}
In the case that  $K_1\geq K_2$ and $K_1$ is a link ($K_2$ could be a knot or link),  $P_{K_2}(u)$  divides one of the four rep-polynomials of $K_1$, $P_{K_1}(u),  P_{K_1}(iu), P'_{K_1}(u), P'_{K_1}(iu).$ 
\end{theorem}
\begin{proof} 
Let $K_2=C[a_1,\cdots,a_m]$ and ${\bf a}=(a_1,\cdots,a_m)$.

 (ii) $\Rightarrow$ (i) :
Since if the rep-polynomial of $K_2=C[a_1,\cdots,a_m]$ is a factor of a rep-polynomial of $K_1=C[n_1,\cdots,n_k]$ then the Riley polynomial of $K_2=S(\alpha_2, \beta_2)$ is a factor of the Riley polynomial of $K_1=S(\alpha_1, \beta_1)$ with $\frac{\beta_1}{\alpha_1}=[n_1,\cdots,n_k]$ and $\frac{\beta_2}{\alpha_2}=[a_1,\cdots,a_m]$, 
  there exists an epimorphism from   $G(K_1)$ to $G(K_2)$ by the result of 
Kitano and Morifuji, Theorem 3.1  in \cite{KM}.

(i) $\Rightarrow$ (ii) :
If we assume that there exists an epimorphism from  $G(K_1)$ to $G(K_2)$, then $K_1$ has an ORS-expansion of type $2n+1$ with respect to any Conway's normal form $K_2=C[\bf a]$, that is, 
$$K_1=C[\epsilon_1{\bf a}, 2c_1, \epsilon_2 {\bf a^{-1}}, 2c_2, \epsilon_3 {\bf a},2c_3, \epsilon_4 {\bf a^{-1}},2c_4,\cdots, \epsilon_{2n} {\bf a^{-1}},2c_{2n}, \epsilon_{2n+1} {\bf a}]$$ 
where
$${\bf a^{-1}}=(a_m, a_{m-1},\cdots,a_1),\ \epsilon_i=\pm 1 (\epsilon_1=1),\ c_i \in \mathbb Z$$ 
by \cite{ALS} and 
\cite{ORS} (see also Theorem 3.1 of \cite{Suzuki}). 	Let $$ {\bf b}=(\epsilon_1{\bf a}, 2c_1, \epsilon_2 {\bf a^{-1}}, 2c_2, \epsilon_3 {\bf a},2c_3, \epsilon_4 {\bf a^{-1}},2c_4,\cdots, \epsilon_{2n} {\bf a^{-1}},2c_{2n}, \epsilon_{2n+1} {\bf a}).$$
Then  $P_{C[\bf a]}(u)$  divides $P_{C[\bf b]}(u)$ by  Theorem \ref{expansion}, which implies that  (ii) holds because $P_{C[\bf b]}(u)$ is either   $P_{K_1}(u)$ or $P'_{C[\bf b]}(u)$. 

The last statement is similarly proved.
 \end{proof}

\begin{corollary}\label{Epi-uPoly}Let  $K_1=S(\alpha,\beta)$ and $K_2=S(\alpha',\beta')$. Then $K_1\geq K_2$ if and only if the Riley polynomial of $S(\alpha',\beta')$ divides  the Riley polynomial of either  $S(\alpha,\beta)$ or $S(\alpha,\beta'')$, where $\beta\beta''\equiv\pm 1\, (\emph{mod}\ \alpha)$.
\end{corollary}
\begin{proof} 
This follows from that if $\frac{\beta}{\alpha}=[n_1,n_2,\cdots,n_k]$ and $\frac{\beta''}{\alpha''}=[n_k,n_{k-1},\cdots,n_1]$, then  $\alpha=\alpha''$ and $\beta\beta''\equiv (-1)^{k-1}$ $(\text{mod}\ \alpha)$.
\end{proof}

\begin{example} We have seen in Example \ref{C[3]-knot} and  Example \ref{C[2,2,5]-knot} that  the rep-polynomial of  $K_1=C[2,2,5]$ and the trefoil $K_2=C[3]$ are 
$$P_{K_1}(u)=u(u-1)(u+1)h(u)h(-u)$$ 
and
$$P_{K_2}=u(u-1)(u+1)$$ respectively, 
where $$h(u)=u^{12}-2u^{10}+u^9+4u^8-u^7-3u^6+3u^5+3u^4-u^3-u^2+2u+1.$$  Therefore $K_1\geq 3_1$ and $K_1$ has an ORS-expansion of type $3$ with respect to ${\bf a}=(3)$, and also has an ORS-expansion of type $3$ with respect to ${\bf a}=(2,-2)$. Actually,  
\begin{equation*}
\begin{split}
K_1&=C[2,2,5]=C[3,-2,6]=C[3,-2,3,0,3]\\
&=C[2,-4,2,-2,2,-2]=C[2,-2,0,-2,2,-2,2,-2].
\end{split}
\end{equation*}
The upside-down diagram of $C[2,2,5]$ is
$$C[5,2,2]=C[6,-2,3]=C[3,0,3,-2,3]\\
=C[2,-2,2,-2,2,0,2,-2],$$
and its rep-polynomial is
$u(u-1)(u+1)\tilde{h}(u)\tilde{h}(-u)$, 
where $$\tilde{h}(u)=u^{12}+2u^{11}-6u^{10}-13u^9+10u^8+27u^7-u^6-19u^5-7u^4+3u^3+5u^2+4u+1.$$
Note that $h(u)\neq \tilde{h}(u)$, but they give the same trace field by Corollary \ref{trace-field}.
\end{example}

\begin{remark}\label{KM-C}
It is an immediate consequence of Theorem \ref{KM} that  
a 2-bridge knot $K$ is minimal if  $K$ has an irreducible Riley polynomial. 
But the converse statement of this is not always true. For example, the double twist knot  $J(4,4)$ is minimal (see Proposition 3.1 of \cite{NST}), but its rep-polynomial is 
$$P_{K}(u)=u(u^3+2u+1)(u^3+2u-1)(u^4+u^3+2u^2+2u+1)(u^4-u^3+2u^2-2u+1), $$ which implies that the Riley polynomial of $J(4,4)$ has two irreducible factors. 
\end{remark}

\appendix
\section{Torus knots and links}\label{EX}  

\subsection{Torus knots}\label{T-Knot}
To compute the rep-polynomial of the torus knot $K=T(2,2k+1)$, we start with two vectors $a_0,b_0$ with  $\langle a_0,b_0\rangle=u$. Then the vectors at the $i$-th step is calculated as follows. (See Figure \ref{TorusLinks}.)
\begin{equation*}
(a_{i},b_{i})=(a_0,b_0)X(u)^{i}=(a_0,b_0)\begin{pmatrix}
-p_{i-1}(-u) & -p_{i}(-u)\\
p_{i}(-u) & p_{i+1}(-u) 
 \end{pmatrix}
\end{equation*}

From this we get the last vectors 
$$a_{2k+1}=-p_{2k}(-u)a_1+p_{2k+1}(-u) b_1$$ and
$$b_{2k+1}=-p_{2k+1}(-u)a_1+p_{2k+2}(-u) b_1.$$
Therefore 
$$P_K(u)=\pm \langle a_0,a_{2k+1}\rangle=\pm up_{2k+1}(-u)=\pm\langle b_0,b_{2k+1}\rangle$$
and thus the solutions of $p_{2k+1}(u)=0$ give non-abelian parabolic representations of $K=T(2,2k+1)$.
By Lemma \ref{cheby-prop}, this equation can be rewritten using $f_n(u):=p_{n+1}(u)-p_n(u)$ as follows.
\begin{equation*}
\begin{split}
p_{2k+1}(u)&=p_{k+1}^2(u)-p_{k}^2(u)\\
&=(p_{k+1}(u)+p_{k}(u))(p_{k+1}(u)-p_{k}(u))\\
&=(-1)^k f_k(u)f_k(-u)=f_k(u)\hat{f}_k(u)\\
&=(-1)^k f_k(2-u^2)=\hat{f}_k(u^2-2).
\end{split}
\end{equation*}
So there is a 1-1 correspondence between the roots of $f_k(u)$ and the  non-abelian parabolic representations of $K=T(2,2k+1)$, and the Riley polynomial is $\pm f_k(2-y)$.

\begin{figure}[hbt]
	\begin{center}
		\scalebox{0.3}{\includegraphics{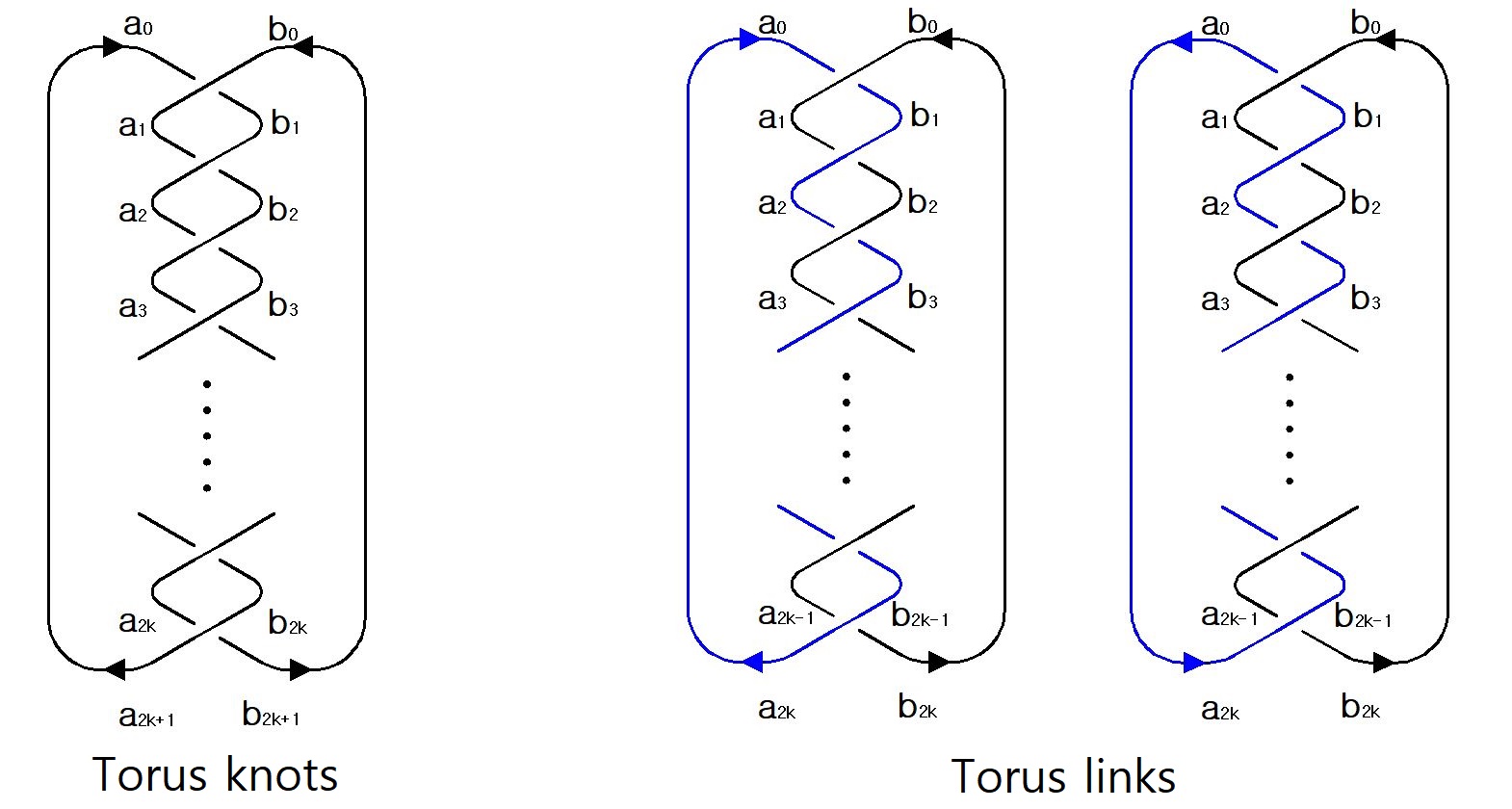}}
	\end{center}
	\caption{Torus knots and Torus links }\label{TorusLinks}
\end{figure}

\subsection{Torus links : $K=T(2,2k)$}
\begin{itemize}
\item Case 1 : The left diagram of Figure \ref{TorusLinks}.

From
\begin{equation*}
(a_{2k},b_{2k})=(a_0,b_0)X(u)^{2k}=(a_0,b_0)\begin{pmatrix}
-p_{2k-1}(-u) & -p_{2k}(-u)\\
p_{2k}(-u) & p_{2k+1}(-u) 
 \end{pmatrix},
\end{equation*}
 we get the last two vectors 
$$a_{2k}=-p_{2k-1}(-u)a_0+p_{2k}(-u) b_0$$ and
$$b_{2k}=-p_{2k}(-u)a_0+p_{2k+1}(-u) b_0.$$
Therefore 
$$P_K(u)=\pm\langle a_0,a_{2k}\rangle=\pm up_{2k}(-u)=\pm\langle b_0,b_{2k}\rangle$$
Note that $u$ is a factor of $p_{2k}(-u)$ and each root of $$\frac{1}{u^2}P_K(u)=\frac{1}{u}p_{2k}(u)=0$$ gives a non-abelian parabolic representation of $K=T(2,2k)$.
Since $p_2(u)=u$, there is no non-abelian parabolic representation for a link $K=T(2,2)$.

\item Case 2 : The right diagram of Figure \ref{TorusLinks}.

If we let $t=-(2+u^2)$, then 
\begin{equation*}
\begin{split}
(a_{2k},b_{2k})&=(a_0,b_0)
(X(u)X(-u))^{k}\\
&=(a_0,b_0)\begin{pmatrix}
-(p_{k-1}(t)+ p_{k}(t))& -up_{k}(t)\\
-up_{k}(t) & (p_{k}(t)+ p_{k+1}(t))
 \end{pmatrix},\\
\end{split}
\end{equation*}
and thus we get the last vectors 
$$a_{2k}=-(p_{k-1}(t)+ p_{k}(t))a_0-up_{k}(t)b_0$$ and
$$b_{2k}= -up_{k}(t)a_0+(p_{k}(t)+ p_{k+1}(t))b_0.$$
Therefore 
$$P_K(u)=\pm\langle a_0,a_{2k}\rangle=\pm u^2p_{k}(t)=\pm u^2p_{k}(-2-u^2)=\pm \langle b_0,b_{2k}\rangle$$
Note that  $p_{k}(-2)=(-1)^{k+1}k\neq 0$ by (ii) of  Lemma \ref{cheby-prop} and this implies that $u$ is not a factor of $p_{k}(-2-u^2)$  and thus all the solutions of $\frac{P_K(u)}{u^2}=0$ give non-abelian parabolic representations of $K=T(2,2k)$.
Since $p_1(-2-u^2)=1$, we can check again that  there is no non-abelian parabolic representation for a link $K=T(2,2)$.
\end{itemize}
\begin{remark}
\begin{enumerate}
\item [(i)] From the above result we see that every torus link has non-abelian parabolic representations if it is not the Hopf link. For example, $K=T(2,4)$ has 2 non-abelian parabolic representations up to conjugate, which correspond to the roots of the equation $u^2 =\pm 2$. Here $u^2-2=\frac{1}{u}p_{4}(u)$ is from Case 1 and $u^2+2= -p_{2}(-2-u^2)$  is from Case 2.
\item [(ii)]
Since  $p_{mn}(u)=p_m(v_n(u))p_n(u), \forall m,n \in \mathbb Z$ by (vii) of Lemma \ref{cheby-prop}, we get 
 $$p_{2k}(u)=p_k(v_2(u))p_2(u)=up_k(v_2(u))$$ and this implies the following: If we denote the two rep-polynomial of $K$ by $P_1(u)=up_{2k}(u)$ and $P_2(u)=u^2p_k(-2-u^2)$, then 
 $$P_2(iu)=-u^2p_k(-2+u^2)=-u^2p_k(v_2(u))=-up_{2k}(u)=-P_1(u)$$
and 
$$P_1(iu)=iup_{2k}(iu)=iu(iu)p_k(v_2(iu))=-u^2p_k(-u^2-2)=-P_2(u).$$  
\end{enumerate}
\end{remark}

\begin{remark}
For each oriented Conway diagram $C[n_1, \cdots, n_k]$ of a  2-bridge link $K$, we can compute easily the rep-polynomial of $K$ 
applying the same procedure to each block as done in the torus link. That is, we get the last two vectors of the $i$-th block, $\epsilon _i a_{i,f}, \epsilon'_i b_{i,f}$, 
multiplying the matrix 
$\begin{pmatrix}
-p_{n_i-1}(-u) & -p_{n_i}(-u)\\
p_{n_i}(-u) & p_{n_i+1}(-u) 
 \end{pmatrix}$ 
to the first two vectors of the $i$-th block, $\epsilon _i a_{i,0}, \epsilon' _i b_{i,0}$. Here $\epsilon _i$ and $\epsilon '_i$ are either $1$ or $i$ depending upon the orientation of the arcs.
\end{remark}

\section{$J(2n_1,2n_2,\cdots,2n_m)$-knots and links}\label{EvenExpansion}  
Each 2-bridge link $K$ has an even expansion $J(2n_1,2n_2,\cdots,2n_m)$, where  $m$ is even if $K$ is a knot and $m$ is odd if $K$ is a link. In this section, we compute the rep-polynomial of such diagrams. 
Using the even expansions, we don't need to consider $\epsilon _i, \epsilon' _i$ for each block and thus we have an explicit formula for the matrix to be multiplied at each block only depending upon whether $i$ is even or odd, that is, whether the block is on the left line or on the  right line.
\begin{figure}[hbt]
	\begin{center}
		\scalebox{0.4}{\includegraphics{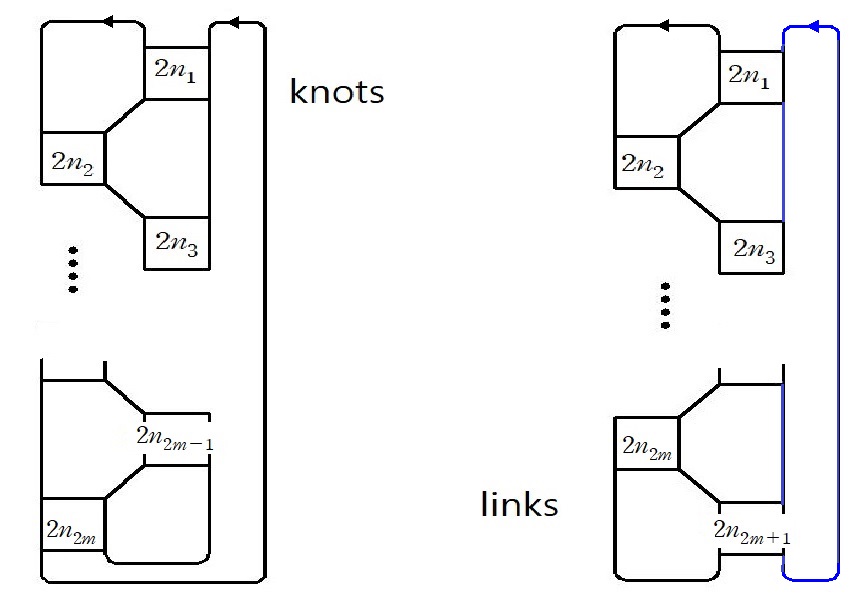}}
	\end{center}
	\caption{even expansion diagram}\label{Even-Exp}
\end{figure}

From our convention of indexing in the diagram, 
$$a_{2,0}=a_{1,0},\ b_{2,0}=a_{1,f}$$ and  for $k\geq 1$ 
$$ a_{2k+1,0}=b_{2k,f},\  b_{2k+1,0}=b_{2k-1,f},\ a_{2k+2,0}=a_{2k,f},\  b_{2k+2,0}=a_{2k+1,f}.$$
If we give an orientation on this diagram, then all the  blocks on the right lines have the same orientation types and the same is true for all the  blocks on the left lines. For example, if the orientation of $b_{1,0}$ is given as in Figure \ref{Even-Exp}, then the orientation of each block will be as in Figure \ref{Odd-Even}.
\begin{figure}[hbt]
	\begin{center}
		\scalebox{0.5}{\includegraphics{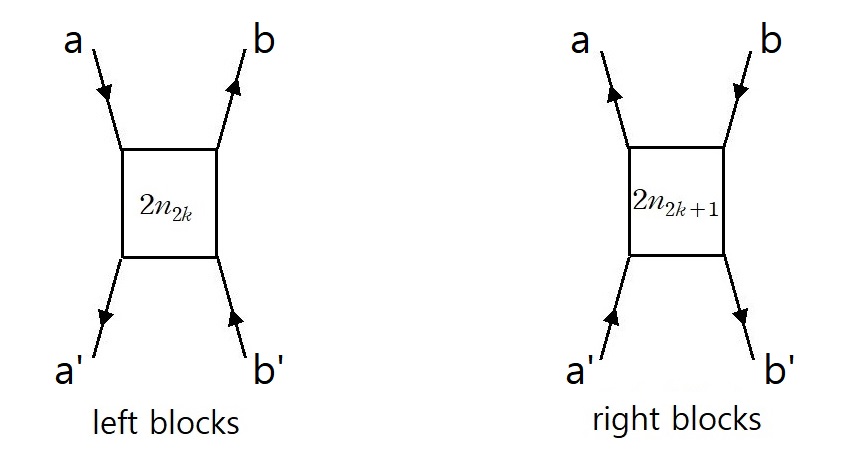}}
	\end{center}
	\caption{}\label{Odd-Even}
\end{figure}

Therefore we can compute arc vectors in each block as follows.
\begin{enumerate}
\item If $a=a_{2k,0},\ b=b_{2k,0}$ and $a'=a_{2k,f},\ b'=b_{2k,f}$, then we get
\begin{equation*}
	\begin{split}
		(a',b')=(a,b)(X(-u)X(u))^n&=(a,b)\begin{pmatrix}
			-(p_{n-1}(t)+ p_{n}(t))& up_{n}(t)\\
			up_{n}(t) & (p_{n}(t)+ p_{n+1}(t))
		\end{pmatrix}\\
         &=(a,b)\begin{pmatrix}
			(-1)^nf_{n-1}(-t))& up_{n}(t)\\
			up_{n}(t) & (-1)^nf_{n}(-t))
		\end{pmatrix}\\
		\end{split}
\end{equation*}
where $u=u_{2k},\ t=t_{2k}=-2-u_{2k}^2,\ n=n_{2k}$.
\item If $a=a_{2k+1,0},\ b=b_{2k+1,0}$ and $a'=a_{2k+1,f},\ b'=b_{2k+1,f}$ and $k>0$, then we get 
\begin{equation*}
	\begin{split}
		(a',b')=(a,b)(X(u)X(-u))^n&=(a,b)\begin{pmatrix}
			-(p_{n-1}(t)+ p_{n}(t))& -up_{n}(t)\\
			-up_{n}(t) & (p_{n}(t)+ p_{n+1}(t))
		\end{pmatrix}\\
         &=(a,b)\begin{pmatrix}
			(-1)^nf_{n-1}(-t))& -up_{n}(t)\\
			-up_{n}(t) & (-1)^nf_{n}(-t))
		\end{pmatrix}\\
		\end{split}
\end{equation*}
where $u=u_{2k+1},\ t=t_{2k+1}=-2-u_{2k+1}^2,\ n=n_{2k+1}$.
\end{enumerate}
By the above process, we obtain the last two vectors $a_{m,f}, b_{m,f}$ and then we get the rep-polynomial $P_K(u)$ of $K$ as follows. If $K$ is a knot, which is the case when $m$ is even, 
$P_K(u)=\pm\langle a_{m,f}, b_{1,0} \rangle$, and $P_K(u)=\pm\langle b_{m,f}, b_{1,0} \rangle$ if $K$ is a link.

\end{document}